%% file: realint_current1.tex
\documentclass[12pt]{amsart}
\usepackage{amssymb, amsfonts,amscd,graphicx}
\usepackage[all]{xy}
\input{macros1}

\author[Cluckers]{Raf Cluckers}
\address{Universit\'e Lille 1, Laboratoire Painlev\'e, CNRS - UMR 8524, Cit\'e Scientifique, 59655
Villeneuve d'Ascq C'edex, France, and, Katholieke Universiteit Leuven, Departement wiskunde,
Celestijnenlaan 200B, B-3001 Leu\-ven, Bel\-gium.
}
\email{raf.cluckers@wis.kuleuven.be}
\urladdr{http://www.wis.kuleuven.be/algebra/Raf/}

\author[Miller]{Daniel~J.~Miller}
\address{Emporia State University, Department of Mathematics, Computer Science and Economics, 1200 Commercial Street, Campus Box 4027, Emporia, KS 66801, U.S.A.}
\email{dmille10@emporia.edu}

\title[Integration of Constructible Functions]{Stability under integration of sums of products of real globally subanalytic functions and their
        logarithms}

\begin{document}

\begin{abstract}
We study Lebesgue integration of sums of products of globally subanalytic functions and their logarithms, called constructible functions. Our first theorem states that the class of constructible functions is stable under integration. The second theorem treats integrability conditions in Fubini-type settings, and the third result gives decay rates at infinity for constructible functions. Further, we give preparation results for constructible functions related to integrability conditions.
\end{abstract}

\maketitle

\input{intro_new1}

\input{subPrep1}

\input{cellPrep1}

\input{sliver1}

\input{goodcoor1}

\input{intprep1}
\input{LRCfamily_plus1}
\input{Section8}

\bibliographystyle{amsplain}
\bibliography{bibliotex}
\end{document}

%% file: macros1.tex
\newtheoremstyle{slthm}
  {9pt}
  {5pt}
  {\slshape}
  {}
  {\bfseries}
  {.}
  {.5em}
  {\thmname{#1} \thmnumber{#2}{\rm \thmnote{ (#3)}}}

\newtheoremstyle{prcl}
  {9pt}
  {5pt}
  {\slshape}
  {}
  {\bfseries}
  {.}
  {.5em}
  {\thmname{#3} \thmnumber{ #2}}

\newtheoremstyle{slthmprime}
  {9pt}
  {5pt}
  {\slshape}
  {}
  {\bfseries}
  {.}
  {.5em}
  {\thmname{#1} \thmnumber{#2}${}^\prime${\rm \thmnote{ (#3)}}}

\theoremstyle{slthm}
\newtheorem{theorem}{Theorem}[section]
\newtheorem{lemma}[theorem]{Lemma}

\newtheorem{proposition}[theorem]{Proposition}

\newtheorem{corollaries}[theorem]{Corollaries}

\theoremstyle{slthmprime}
\newtheorem{Theorem}{Theorem}[section]

\theoremstyle{definition}
\newtheorem{definition}[theorem]{Definition}
\newtheorem{definitions}[theorem]{Definitions}

\newtheorem{example}[theorem]{Example}

\theoremstyle{remark}

\newtheorem{notation}[theorem]{Notation}

\newtheorem{remark}[theorem]{Remark}

\renewcommand{\labelenumi}{\rm \theenumi.}

\newenvironment{renumerate}
        {\renewcommand{\labelenumi}{\rm (\roman{enumi})}
         \begin{enumerate}}
        {\end{enumerate}}

\numberwithin{equation}{section}

\allowdisplaybreaks[2]

\newcommand{\tprime}[1]{\ensuremath{{#1}^\prime}}

\newcommand{\A}{\mathcal{A}}
\newcommand{\B}{\mathcal{B}}
\newcommand{\C}{\mathcal{C}}

\newcommand{\F}{\mathcal{F}}

\renewcommand{\H}{\mathcal{H}} %

\renewcommand{\L}{\mathcal{L}}   %

\renewcommand{\P}{\mathcal{P}} %

\renewcommand{\S}{\mathcal{S}}   %

\newcommand{\NN}{\mathbb{N}}

\newcommand{\QQ}{\mathbb{Q}}
\newcommand{\RR}{\mathbb{R}}

\newcommand{\PP}{\mathbb{P}}

\renewcommand{\bar}[1]{\overline{#1}} %
\newcommand{\tld}[1]{\widetilde{#1}}

\DeclareMathOperator{\im}{im}

\newcommand{\Restr}[1]{\big|_{#1}}

\newcommand{\PD}[3]{\frac{\partial^{#1}#2}{\partial {#3}^{#1}}}

\DeclareMathOperator{\cl}{cl}
\DeclareMathOperator{\bd}{bd}

\newcommand{\barellprime}{\bar{\ell}^{\,\prime}}

%% file: intro_new1.tex
\section{Introduction}

Several ``tame'' frameworks on real affine spaces have come up through, for example,
semialgebraic and subanalytic geometry, and more generally, o-minimal structures, where tame means that many concrete analytic, topological, and geometric results hold.  On the other hand, many mathematical mysteries and non-tame behavior exist in the theory of real integrals, and they can be studied through concepts like $L^p$-spaces, Hardy spaces, periods,
and so on. For example, integrals in o-minimal set-ups are very little understood, with probably one of the hardest and most important cases being the study of integrals of semialgebraic functions defined over $\QQ$, which are called periods; see \cite{KontZag}. To give another example, for $f:\RR^n\times\RR^m\to\RR$ a function definable in an o-minimal structure  (for example, a semialgebraic function), the function given by the parameterized integral $x\in\RR^n \mapsto\int_{\RR^m}f(x,y)dy$ (a ``family'' of periods) will in general not belong to the original o-minimal structure. In this paper we
study such families of integrals in the o-minimal framework of globally subanalytic functions, which is richer than the collection of semialgebraic functions, but is still rather close to it.

In theory one could try to  expand any o-minimal structure to a larger structure by closing under parameterized integration, or,
under taking integral manifolds of suitable definable $1$-forms, as in the Pfaffian closure construction in \cite{SpeissPfaff}, but then the new structure will in general be much larger and more complicated than the original structure.
In this paper we follow another philosophy. Our central player is Lebesgue integration, and we put aside irrelevant definability constructions, like taking inverse functions of bijections, taking compositions, and so on. For example, the logarithm can be created using a parameterized integral, but we will not encounter the exponential function as a parameterized integral of globally subanalytic functions. The aim of the alternative philosophy is to find a class of functions with a tame theory of integration and
nice geometrical properties,
while avoiding the complexity to be blown up.
In this paper we provide such a tame theory of integration
for globally subanalytic sets and functions. We show that it is enough to enlarge this class of functions with (obligatory) logarithms to obtain stability and related properties under parameterized integration; see Theorems \ref{thm:LRC}, \tprime{\ref{THM:LRCfamily}} and Proposition \ref{thm:decayRate}.
As far as we can see, it is the only known framework of this kind, and it shows that the globally subanalytic functions form a very special collection of functions. Note that inspiring partial results in this direction were obtained by Comte, Lion, and Rolin
\cite{LR98}, \cite{CLR2000}, and that the $p$-adic and motivic understanding of parameterized integrals predates the study in the real setting through the work \cite{Denef1}, \cite{Ccell}, \cite{CLoes}, \cite{CLR}.

Our main Theorem \ref{thm:LRC} on the stability under parameterized integration deals crudely with integrability issues, and thus it must be understood together with Theorem \tprime{\ref{thm:LRCfamily}} on integrability.
These two theorems form a powerful pair,
particularly in view of Fubini's Theorem, and they are fundamental results on Lebesgue integration in the tame but rich context of globally subanalytic functions and their logarithms. The preparation theorems we obtain for the class of so-called constructible functions and Proposition \ref{thm:decayRate} describe the tame geometry of this class of globally subanalytic functions and their logarithms; see Definition \ref{def:constru}, in which we omit the word ``globally'' after the convention made in  Section \ref{sec:MainResults}.

\subsection{The Main Results}\label{sec:MainResults}

A function $f:X\subset \RR^n\to\RR$ is called globally subanalytic if its graph is a globally subanalytic set, and a set $A\subset\RR^n$ is called globally subanalytic if its image under the natural embedding of $\RR^n$ into $n$-dimensional real projective space, namely $\RR^n\to\PP^n(\RR): (x_1,\ldots,x_n)\mapsto(1:x_1:\cdots:x_n)$, is a subanalytic subset of $\PP^n(\RR)$ in the classical sense.  This condition on $A$ is equivalent to saying that $A$ is definable in the expansion of the real field by all restricted analytic functions $f:\RR^k\to\RR$, with $k\in\NN$, where a restricted analytic function $f$ is, by definition, associated to an analytic function $f_0$ on an open $U\subset\RR^k$ containing $[-1,1]^k$ and defined by
$$
f:\RR^k\to\RR:x\mapsto \left\{\begin{array}{ll} f_0(x) & \mbox{ if } x\in  [-1,1]^k,\\
0 & \mbox{ otherwise. }  \end{array}\right.
$$

Henceforth, in the whole paper, we abbreviate the terminology and consequently use the word \textbf{subanalytic} to mean globally subanalytic, for both sets and functions.  It is important to note that with this convention on our terminology, the natural logarithm $\log:(0,+\infty)\to\RR$ is not subanalytic.  (For more background on subanalytic sets and functions, see Bierstone and Milman \cite{BM88}, Denef and van den Dries \cite[Section 4]{DvdD}, and van den Dries and C. Miller \cite[2.5.\,Examples(4)]{vdDcM}.)

\begin{definition}
For $X$ a subanalytic set, $m\geq 0$ an integer, and $f:X\times \RR^m\to\RR$ a Lebesgue measurable function, consider the parameterized integral
\[
I_{X}(f):X\to\RR:x\mapsto
\left\{
\begin{array}{ll}
\int_{\RR^m}f(x,y)dy,
    & \mbox{if $f(x,\cdot)$ is integrable for each $x\in X$,}\\
0,
    & \mbox{otherwise},
\end{array}
\right.
\]
with $dy=dy_1\wedge \ldots\wedge dy_m$ the Lebesgue measure on $\RR^m$ and $f(x,\cdot)$ the function $y\in\RR^m\mapsto f(x,y)$, and where integrable means Lebesgue integrable.
\end{definition}
Theorem \ref{thm:LRC} will show that the following question has a particularly nice answer.
\begin{quote}
What functions can be obtained starting from the subanalytic functions through the finite application of the integral operators $I_X$?
\end{quote}
The answer will be in terms of the following rings of constructible functions.

\begin{definition}\label{def:constru}
For each subanalytic set $X$, let $\C(X)$ be the $\RR$-algebra of real-valued functions on $X$ generated by all subanalytic functions on $X$ and the functions $x\mapsto \log f(x)$, where $f:X\to(0,+\infty)$ is subanalytic.  We call $f\in \C(X)$ a constructible function on $X$, and $\C(X)$ is called the algebra of {\bf constructible functions} on $X$.
\end{definition}

(In some literature  constructible functions are named log-analytic functions.)
It is easy to see that all constructible functions can be obtained from the subanalytic functions through a single application of the $I_X$.  Indeed, for $f\in\C(X)$,  there are subanalytic functions $f_i:X\to\RR$ and $f_{i,j}:X\to(0,+\infty)$ such that
\[
f(x) = \sum_{i=1}^{k} f_i(x)\prod_{j=1}^{l_i}\log f_{i,j}(x).
\]
By absorbing sign information inside the functions $f_i$, we may assume that $f_{i,j}(x) > 1$ for all $x$.  Therefore if we write
\[
f(x) = \sum_{i=1}^{k} \int_{1}^{f_{i,1}(x)}\!\!\!\cdots\int_{1}^{f_{i,l_i}(x)}
    \frac{f_i(x)}{t_{i,1}\cdots t_{i,l_i}} dt_{i,1}\cdots dt_{i,l_i},
\]
and if we extend each integrand $(x,t_{i,1},\ldots,t_{i,l_i})\mapsto \frac{f_i(x)}{t_{i,1}\cdots t_{i,l_i}}$ to $X\times\RR^{l_1+\cdots+l_k}$ by zero outside its natural domain, we see that $f = I_X(F)$ for a single subanalytic function $F:X\times\RR^{l_1+\cdots+l_k}\to\RR$.  It follows from Theorem \ref{thm:LRC} and the above observation that the constructible functions are in fact \emph{precisely} the functions obtainable from the subanalytic functions through the  $I_X$
because the constructible functions are stable under integration.

\begin{theorem}[Stability under integration]\label{thm:LRC}
Let $f$ be in $\C(X\times \RR^{m})$ for some subanalytic set $X$. Then $I_{X}(f)$ is in
$\C(X)$.
\end{theorem}

The definition of $I_X$ deals crudely with integrability issues.  So in order to apply Theorem \ref{thm:LRC} in combination with Fubini's theorem, we show that for constructible functions, the notion of ``integrable almost everywhere'' has a subanalytic interpretation.  In simplest form this means the following.



\begin{theorem}[Integrability]\label{thm:LRCfamily}
Let $f$ be in $\C(\RR^{n+1})$.
Suppose that for  all $y$ in a dense subset $C$ of $\RR^n$, the function
\begin{equation*}\label{eq:z}
f(y,\cdot):\RR\to \RR: z\mapsto f(y,z)
\end{equation*}
is integrable.
Then there exists a subanalytic dense subset $C'$ of $\RR^n$ such that $f(y,\cdot)$ is integrable over $\RR$ for all $y\in C'$.
\end{theorem}

Note that if $\RR^n\setminus C$ has measure zero, then $C$ is dense in $\RR^n$.  Also note that a subanalytic set $C'$ is dense in $\RR^n$ if and only if $\RR^n\setminus C'$ has measure zero, if and only if the dimension of $\RR^n\setminus C'$ is less than $n$.

We do not provide a separate proof of Theorem \ref{thm:LRCfamily}, but instead prove Theorem \tprime{\ref{THM:LRCfamily}}, which is a stronger, parameterized version of Theorem \ref{thm:LRCfamily}.  The proof of Theorem \ref{thm:LRC} will be reduced to the case that $m=1$ by Theorem \tprime{\ref{THM:LRCfamily}} (with $Y = \RR^{m-1}$) and Fubini's
theorem.

\addtocounter{theorem}{-1}
\setcounter{Theorem}{\value{theorem}}
\addtocounter{theorem}{1}

\begin{Theorem}[Integrability in families]\label{THM:LRCfamily}
Let $f$ be in $\C(X\times Y\times \RR)$ for some subanalytic sets $X$ and $Y$.
Suppose that for each $x\in X$ and for all $y$ in a dense subset $C_x$ of $Y$, the function
\begin{equation*}\label{eq:z:cor}
f(x,y,\cdot):\RR \to \RR: z\mapsto f(x,y,z)
\end{equation*}
is integrable, where $C_x$ may depend on $x$.
Then there are dense subsets $C'_x$ of $Y$ such that the set $\{(x,y)\in X\times \RR^n : y\in C'_x\}$ is subanalytic and such that for all $x\in X$ and all $y$ in $C'_x$, the function $f(x,y,\cdot)$ is integrable.
\end{Theorem}

Along the way we prove the following
proposition for bounding
constructible functions
at infinity.

\begin{proposition}[Decay rates]\label{thm:decayRate}
Let $f$ be in $\C(X\times\RR)$ for some subanalytic set $X$, and suppose that
\[
\lim_{y\to +\infty} f(x,y) = 0
\]
for all $x\in X$.  Then there exist a constant $r>0$ and a subanalytic function $g:X\to(0,+\infty)$ such that
\[
|f(x,y)| \leq y^{-r}
\]
for all $x\in X$ and all $y$ with $y > g(x)$.
\end{proposition}

Theorem \ref{thm:LRC} yields a small and geometrically interesting framework of functions which is closed under integration and which is rich enough to contain all
semialgebraic functions. Thus, this framework contains a relative (in the sense of parameterized) version of the periods as presented by Kontsevich and Zagier \cite{KontZag}.
Note that Theorems \ref{thm:LRC} and \ref{thm:LRCfamily} in the special case that $f$ is a subanalytic function appear already in \cite{CLR2000}. In \cite{LR98}, the $I_X(f)$ of Theorem \ref{thm:LRC} for constructible $f$ are shown to belong to a larger class of functions than $\C(X)$, hence our Theorem \ref{thm:LRC} is more precise and gives moreover a stable framework under integration.

In \cite{CluckersMiller2} we study $L^p$-properties of constructible functions for general $p$, as well as variants of Theorem \tprime{\ref{THM:LRCfamily}}.  In \cite{CluckersMiller3} we give an application of our work to harmonic analysis and generalized Fourier transforms in the context of chapter VIII of Stein's book \cite{Stein}.

\subsection{Some context}
It is classical to study asymptotic expansions for one parameter integrals like
$$
E(z):=\int\limits_{z=f(x),\ x\in (-1,1)^n} \frac{dx_1\wedge\ldots\wedge dx_n}{df},
$$
or its Fourier transform,
with $f:U\to\RR $ analytic on an open set $U$ containing $[-1,1]^n$, having some isolated critical points in $(-1,1)^n$, and where $dx/df$ is the Gelfand-Leray differential form and $z$ runs over noncritical values of $f$. Such asymptotic expansions are  interesting for $z$ going to a critical value of $f$ and are used, for example, in Varchenko's definition of the Hodge filtration and the spectrum of an isolated singularity of $f$; see \cite{AGZV88} and \cite{Kulikov98}. Just to name some other studies of similar parameterized integrals we mention the link made to monodromy by Malgrange \cite{Mal74}, and the work for nonisolated critical points by Barlet \cite{Barlet82} and Loeser \cite{Loeser86}.

Instead of looking at such classical asymptotical expansions with only one parameter,
in Theorem \ref{thm:LRC}
we describe globally (that is, for all parameters) what happens for a richer class of integrands (namely, what we call constructible functions) and when arbitrarily many parameters are involved.

Lion and Rolin \cite{LR98}, in their Th\'eor\`eme d'Int\'egration, examined bounded constructible functions (with bounded domain and bounded range) and described the parameterized integrals of such bounded functions, but did not show that
they are
again constructible. Instead they only showed that such an integral is a pointwise limit of constructible functions. Combined with their Th\'eor\`eme d'\'Elimination \cite{LR98}, they also showed that such an integral is piecewise the restriction of constructible functions, but possibly with pieces more general than subanalytic sets, which again misses the strength and naturality of Theorem \ref{thm:LRC}.

Comte, Lion and Rolin \cite{CLR2000} continued the work of \cite{LR98} by showing that the ($k$-dimensional) volume of sets in a subanalytic family is a constructible function in the parameters of the family, provided that the volumes are finite. This indeed is the  special case of Theorem \ref{thm:LRC} where $f$ itself is (globally, as always) subanalytic, as opposed to constructible (the latter may involve logarithms). They also show in Th\'eor\`eme 1' of \cite{CLR2000} that the set of parameters where the volume is finite is subanalytic, which compares to Theorem \ref{thm:LRCfamily} but which does not generalize literally to constructible functions, as is shown by the following example.


\begin{example}\label{ex:notint}
Let $g:\RR\to \RR$ be a function in $\C(\RR)$ that is not
integrable over $\RR$. Define
$f:(\RR\setminus\{0\})^3\to\RR:(x,y,z)\mapsto (x-\log|y|)g(z)$. Then
the set of $(x,y)$ such that $z\mapsto f(x,y,z)$ is integrable is
not (globally, as always in this paper) subanalytic since it is the graph of $\log|\cdot|$.
\end{example}
This example shows that one can not really hope for more than what is given by Theorem
\ref{thm:LRCfamily} and that a naive adaptation of Theorem 1' of \cite{CLR2000} to the context of constructible functions is false.
%
%
%
Because $x\mapsto \log f(x)$ is in general not a (globally) subanalytic function for $f$ a subanalytic function, the assumption in Theorem \ref{thm:LRCfamily} that the set $C$ is dense   is
essential.


Logarithms play an important role in many studies of integrals and differential equations, but it is somehow surprising that they suffice to yield a  framework closed under integration as given by Theorem \ref{thm:LRC}, although \cite{CLR2000} and the $p$-adic situation already hinted in that direction.

\subsection{The $p$-adic analogue} 
Using Haar measures on the $p$-adic numbers, Igusa  \cite{Igusa3} gave several
asymptotic expansions for $p$-adic one parameter integrals, and
Denef \cite{Denef3}, \cite{Denef1} obtained results which are very close to the $p$-adic analogue of Theorem \ref{thm:LRC}.
 Similar results to our two main theorems were made uniformly in big $p$ and for motivic
integrals in \cite{CLoes}. These steps meant a global description
instead of asymptotic expansions, arbitrarily many parameters in
the integrals, and a framework closed under integration. In \cite{Ccell} and \cite{CLR}, one
enriched the class of integrands to include $p$-adic subanalytic functions, instead of only $p$-adic semialgebraic functions.

In some sense, the $p$-adic situation is easier since many $p$-adic integrals (e.g., of $p$-adic constructible functions) can be reduced to certain sums over the
residue field and the value group, and these are often more easy to handle. On the reals such a reduction lacks, and one has to use very precise and specific versions of preparation theorems to prepare the integrands on nice pieces in a finite partition.

Another difficult and open problem on the real numbers is that, while for the $p$-adics integration is well understood as soon as one has a nice cell decomposition and such cell decomposition is known for semialgebraic as well as for subanalytic set-ups \cite{Denef2}\cite{Ccell}, for the reals a theory of functions closed under integration for now only makes sense in the subanalytic set-up, and is completely open in the semialgebraic set-up. This is so because many non-semialgebraic functions come up as parameterized integrals of semialgebraic functions and not only their logarithms play a role.

\subsection{Our Method}
The heart of our proof of Theorems \ref{thm:LRC}, \tprime{\ref{thm:LRCfamily}} and Proposition \ref{thm:decayRate} lies in the preparation theorem \ref{thm:constrPrep} for constructible functions and our main proposition \ref{prop:constrPrep} about compatibility of such preparations with integrability conditions.  More specifically, in Theorem \ref{thm:constrPrep} we consider a constructible function $f:X\to\RR$ for some subanalytic set $X\subset\RR^n$, and we construct a subanalytic cell decomposition of $X$ such that on each cell $A$ in the decomposition, $f$ can be expressed as a finite sum $f(x) = \sum_{i\in I} T_i(x)$ where each term $T_i(x)$ is a very basic term with good properties, namely, roughly a product of a rational monomial, an integral logarithmic monomial, and a subanalytic unit, with distinct powers in the logarithmic monomials. In Proposition \ref{prop:constrPrep} we then show, for simplicity say when $A=X$,  that if the set of $x_{<n}:=(x_1,\ldots,x_{n-1})$ for which $f(x_{<n},\cdot):x_n\mapsto f(x_{<n},x_n)$ is integrable is dense in $\Pi_{n-1}(A) := \{x_{<n} : x\in A\}$, then for all $x_{<n}\in\Pi_{n-1}(A)$ and all $i\in I$ the function $T_i(x_{<n},\cdot)$
is integrable.  Theorem \tprime{\ref{thm:LRCfamily}} and Proposition \ref{thm:decayRate} will follow. Together with Fubini's Theorem, this will reduce the proof of Theorem \ref{thm:LRC} to the case $m=1$, which will be proven by integrating each of the terms $T_i$
using an adapted version of the procedure of Lion and Rolin \cite{LR98}.
Note that these proofs avoid taking limits for improper integrals, which were used previously in
\cite{LR98}, \cite{CLR2000}.

We begin by establishing refined versions of the Lion-Rolin preparation theorem
of \cite{LR97} for subanalytic functions
which are used to prove Theorem \ref{thm:constrPrep}; see Theorems \ref{thm:cellPrep} and \ref{thm:iteratedCellPrep} below.
To prove the main Proposition \ref{prop:constrPrep} we use an additional, new technique: we introduce special functions, called ``sliver functions'', to study the relative asymptotic behavior of the terms $T_i$ to show that they do not cancel each other out on some thin, open ``sliver''.

\subsection*{Acknowledgment} The authors would like to thank several persons for the interest they have shown in our work and for stimulating discussions, among which M.~Aschenbrenner, D.~Bertrand, J.-B.~Bost, G.~Comte, J.~Denef, Z.~Denkowska, E.~Hrushovski, J.-M.~Lion, F.~Loeser, J.-P. Rolin, P.~Speissegger, and B.~Totaro. Further they would like to thank the Newton Institute for its hospitality and for providing excellent working conditions. 

%% file: subPrep1.tex
\section{Cylindrical preparation of subanalytic functions}\label{s:subPrep}

%
%
In this section we recall the statement of the
subanalytic (cylindrical) preparation theorem from \cite{LR97} (see also
\cite{DJMprep}) and its supporting definitions, and we fix some notation to be used throughout the paper.

\begin{notation}\label{notation:projections}
Consider integers $m$ and $n$ with $0\leq m\leq n$ and a tuple of variables $x=(x_1,\ldots,x_n)$.  For any increasing map $\lambda : \{1,\ldots,m\}\to\{1,\ldots,n\}$, define the coordinate projection $\Pi_\lambda:\RR^n\to\RR^m$ by $\Pi_\lambda(x):=(x_{\lambda(1)},\ldots,x_{\lambda(m)})$, with the understanding that when $m=0$ one uses the conventions that $\RR^0 := \{0\}$ and $\Pi_\emptyset(x) := 0$, where $\emptyset$ denotes the empty map.  An important special case is when $\lambda(i) = i$ for all $i\in\{1,\ldots,m\}$, and in this case we will write $\Pi_m$ instead of $\Pi_\lambda$.  (Thus $\Pi_0 = \Pi_\emptyset$.)  We also use the notation $x_{\leq m} = (x_1,\ldots,x_m)$, $x_{<m} = (x_1,\ldots,x_{m-1})$ and $x_{>m} = (x_{m+1},\ldots,x_n)$.  This notation can be applied to components of maps as well, for example, if $f = (f_1,\ldots,f_n)$ for some real-valued functions $f_i$, then $f_{\leq m} = (f_1,\ldots,f_m)$.  If $A\subset\RR^n$ and $0\leq m \leq n$, we write $A_{x_{\leq m}} = \{x_{>m} : (x_{\leq m}, x_{>m})\in A\}$ for the fiber of $A$ over $x_{\leq m}\in\RR^m$. Further, we write $\im(g)$ for the image of a function $g$.
\end{notation}

\begin{definitions}\label{def:subPrep0}
Call a function $f:X\subset \RR^\ell\to \RR^k$ {\bf analytic} if it
extends to an analytic function on an open neighborhood of $X$.
An analytic function $u:A\to\RR$ on a set $A\subset\RR^n$ is a {\bf unit on $A$} if either $u(x) > 0$ on $A$ or $u(x) < 0$ on $A$.
A {\bf restricted analytic
function} is a function $f:\RR^n\to\RR$ such that the restriction of
$f$ to $[-1,1]^n$ is analytic and $f(x)=0$ on
$\RR^n\setminus[-1,1]^n$.

Recall from the introduction that we call a set or a function  {\bf subanalytic} if and only if it is definable in the expansion of the real field by all restricted analytic functions.  Thus in this paper, ``subanalytic'' is an abbreviation of ``globally subanalytic'', and in this meaning, the natural logarithm $\log:(0,+\infty)\to\RR$ is not subanalytic.

A {\bf subanalytic term} is a function which can be constructed as a
finite composition of restricted analytic functions, the
algebraic operations of addition and multiplication, and
the rational power functions which are defined on $\RR$ by
\[
x\mapsto
\begin{cases}
x^r & \text{if $x\geq 0$},\\
0   & \text{if $x < 0$},
\end{cases}
\]
for each rational number $r$.  We shall also use this terminology for restrictions of subanalytic terms to subanalytic sets.
\end{definitions}

For the rest of the section we fix an ordered list of variables $x_1,\ldots,x_{n+1}$, where $n\geq 0$, and we write $x$ for $(x_1,\ldots,x_n)$ and write $y$ for $x_{n+1}$, since the variable $x_{n+1}$ will play a special role.

\begin{definitions}\label{def:subPrep}
A set $A\subset\RR^{n+1}$ is a {\bf subanalytic cylinder} if
\begin{renumerate}
\item
$n = 0$, and $A$ is of one of the following four forms:
\begin{description}
\item[Form 1]
$A = \{a\}$,

\item[Form 2]
$A = (a,+\infty)$,

\item[Form 3]
$A = (-\infty,a)$,

\item[Form 4]
$A = (a,b)$,
\end{description}
where $a<b$ are real numbers, or

\item
$n > 0$, and $B = \Pi_{n}(A)$ is a subanalytic set defined in a quantifier-free manner using subanalytic terms and the relations $=$ and $<$, and $A$ is of one of the following four forms:
\begin{description}
\item[Form 1]
$A = \{(x,y)\in B\times\RR : y = a(x)\}$,

\item[Form 2]
$A = \{(x,y)\in B\times\RR : y > a(x)\}$,

\item[Form 3]
$A = \{(x,y)\in B\times\RR : y < a(x)\}$,

\item[Form 4]
$A = \{(x,y)\in B\times\RR : a(x) < y < b(x)\}$,
\end{description}
where $a,b:B\to\RR$ are analytic subanalytic terms and $a(x) < b(x)$ on $B$. We call $B$ the {\bf base} of $A$.
\end{renumerate}
We say that $A$ is {\bf thin} (in $y$) if it is of the Form 1 and that $A$ is {\bf fat} (in $y$) if it is of the Form 2, 3, or 4.

If $A$ is a fat subanalytic cylinder with base $B$, a {\bf center
for $A$} is an analytic subanalytic term $\theta:B\to\RR$ whose
graph is disjoint from $A$.

If $A$ is a fat subanalytic cylinder and $\theta$ is a center for $A$, a {\bf
strong subanalytic unit on $A$ with center $\theta$} is a function
$u:A\to\RR$ of the form $u = U\circ\varphi$, where
$\varphi:A\to\RR^N$ is a bounded function for some natural number $N$, $U$ is an analytic unit on the
closure of the image of $\varphi$, and $\varphi$ has the form
\[
\varphi(x,y) =
(a_1(x)|y-\theta(x)|^{r_1},\ldots,a_N(x)|y-\theta(x)|^{r_N}),
\]
where $a_1,\ldots,a_N$ are analytic subanalytic terms on the base of $A$
and $r_1,\ldots,r_N$ are rational numbers (some, or all, of which may be equal to $0$).

For any real-valued functions $f$ and $g$ on a
set $A$, we say that {\bf $f$ is  equivalent to $g$ on $A$},
written $f\sim  g$ on $A$, if there exists $\epsilon>1$ such that for all $x\in A$,
\[
\begin{cases}
\epsilon^{-1}f(x)\leq g(x) \leq \epsilon f(x), &  \text{if
$f(x)\geq 0$},\\
\epsilon f(x) \leq g(x) \leq \epsilon^{-1} f(x), &  \text{if $f(x) <
0$}.
\end{cases}
\]
We write $f\sim_\epsilon g$ on $A$ if we want to specify $\epsilon$.

A partition $\P$ of a set $X$ is called {\bf compatible} with a set $\S$ of
subsets of $X$ if for all $P\in\P$ and all $S\in\S$ either
$P\subset S$ or $P\cap S = \emptyset$.
\end{definitions}

\begin{theorem}[Subanalytic Preparation Theorem \cite{LR97}, \cite{DJMprep}]\label{thm:subPrep}
Let $\F$ be a finite set of real-valued subanalytic functions on a
subanalytic set $X\subset\RR^{n+1}$, let $\S$ be a finite set of
subanalytic subsets of $X$, and let $\epsilon
> 1$.  There exists a finite partition of $X$ into subanalytic
cylinders which is compatible with $\S$ and is such that the
following hold for each cylinder $A$ in this partition:
\begin{renumerate}{\setlength{\itemsep}{3pt}
\item
If $A$ is thin, then for each $f\in\F$ there exists an analytic
subanalytic term $t:B\to\RR$ such that $f(x,y) = t(x)$ on $A$.

\item
If $A$ is fat, then there exists a center $\theta$ for $A$ such that
each $f\in\F$ can be written in the form
\[
f(x,y) = a(x)|y-\theta(x)|^r u(x,y)
\]
on $A$, where $a$ is an analytic subanalytic term on the base of $A$, $r$ is a
rational number, and $u$ is a strong subanalytic unit on $A$ with
center $\theta$.  Moreover, if $\theta$ is not identically zero,
then $y\sim_\epsilon \theta$ on $A$.
}
\end{renumerate}
\end{theorem}

\begin{remark}\label{rmk:unbddFiberCylinder}
In Theorem \ref{thm:subPrep} the requirement that $y\sim_\epsilon \theta$ on $A$ if $\theta$ is not identically $0$ implies that if each of the fibers of $A$ over $\Pi_n(A)$ are unbounded (namely, $A$ is of the Form 2 or 3 of Definitions \ref{def:subPrep}), then $\theta = 0$.
\end{remark}

In addition to the subanalytic preparation theorem,
we will need the following two lemmas.  The first is a reformulation of Lemma 3.4 from \cite{DJMprep}, and the second is a strengthening of Lemma 4.6 from \cite{DJMprep}.

\begin{lemma}[\cite{DJMprep}]\label{lemma:strongUnit}
Let $A\subset\RR^{n+1}$ be a subanalytic cylinder with center $\theta$, let $u$ be a strong subanalytic unit on $A$ with center $\theta$, and let $r$ be a rational number.  There exists a finite partition of $\Pi_n(A)$ into subanalytic cylinders such that for each cylinder $B$ in the partition there exist a natural number $N$ and functions $\varphi:A\cap(B\times\RR)\to\RR^{N+2}$ and $U:\RR^{N+2}\to\RR$
such that $u(x,y) = U\circ\varphi(x,y)$ on $A\cap(B\times\RR)$, $U$ is an analytic unit on the closure of the image of $\varphi$, and $\varphi$ is a bounded function of the form
\[
\varphi(x,y) =
(a_1(x),\ldots,a_N(x),b(x)|y-\theta(x)|^{1/p},c(x)|y-\theta(x)|^{-1/p}),
\]
where $a_1,\ldots,a_N,b,c$ are analytic subanalytic terms on $B$
and $p$ is a positive integer such that $rp$ is an integer.
\end{lemma}

\begin{lemma}\label{lemma:simulPrep}
Let $\theta_1$ and $\theta_2$ be real-valued subanalytic terms
on $\Pi_n(X)$ for a subanalytic set $X\subset\RR^{n+1}$.  There
exists a finite partition of $X$ into subanalytic cylinders such
that for each fat cylinder $A$ in this partition, $\theta_1$ and
$\theta_2$ are both centers for $A$ such that $\theta_1 - \theta_2$ has a constant sign on $\Pi_n(A)$,
and when $\theta_1\neq\theta_2$ on $\Pi_n(A)$, at least one of the following three cases hold on $A$, where the expressions in square brackets are strong subanalytic units on $A$ (with center $\theta_2$ in case (i), and center $\theta_1$ in cases (ii) and (iii)):

\begin{renumerate}{\setlength{\itemsep}{3pt}
\item
$\displaystyle y - \theta_1(x) = (\theta_2(x) - \theta_1(x))\cdot
\left[1 + \frac{y-\theta_2(x)}{\theta_2(x) - \theta_1(x)}\right]$,

\item
$\displaystyle y - \theta_2(x) = (\theta_1(x) - \theta_2(x))\cdot
\left[1 + \frac{y-\theta_1(x)}{\theta_1(x) - \theta_2(x)}\right]$,

\item
$\displaystyle
y - \theta_2(x) = (y - \theta_1(x))\cdot \left[1 +
\frac{\theta_1(x) - \theta_2(x)}{y-\theta_1(x)}\right]$.
}
\end{renumerate}
\end{lemma}

\begin{proof}
By partitioning $X$ we may focus on a fat cylinder $A$ such that $\theta_1 - \theta_2$ has constant sign on $\Pi_n(A)$.  We are done if $\theta_1 = \theta_2$ on $\Pi_n(A)$, and the cases $\theta_1 > \theta_2$ and $\theta_1 < \theta_2$ can be handled similarly, so we assume that $\theta_1 > \theta_2$ on $\Pi_n(A)$.  Choose constants $a$ and $b$ such that $\frac{1}{2} < a < 1 < b < 1+a$, and consider the following sets, each of which is a finite union of cylinders:
\begin{eqnarray*}
A_1
    & = &
    \left\{(x,y)\in A : \frac{y - \theta_2(x)}{\theta_1(x) - \theta_2(x)} > b\right\},\\
A_2
    & = &
    \left\{(x,y)\in A : \left|\frac{y - \theta_1(x)}{\theta_1(x) - \theta_2(x)}\right| < a\right\}, \\
A_3
    & = &
    \left\{(x,y)\in A : \left|\frac{y - \theta_2(x)}{\theta_1(x) - \theta_2(x)}\right| < a\right\}, \\
A_4
    & = &
    \left\{(x,y)\in A : \frac{y - \theta_1(x)}{\theta_1(x) - \theta_2(x)} < -b\right\}. \\
\end{eqnarray*}
By the choice of $a$ and $b$, these sets cover $A$.  The expression in square brackets in (i) is clearly a strong subanalytic unit on $A_3$, and likewise for (ii) on $A_2$.  A little algebra shows that
\[
0 < \frac{\theta_1(x) - \theta_2(x)}{y - \theta_1(x)} < \frac{1}{b-1}
\]
on $A_1$, and that
\[
-\frac{1}{b} < \frac{\theta_1(x) - \theta_2(x)}{y - \theta_1(x)} < 0
\]
on $A_4$.  Therefore the expression in square brackets in (iii) is a strong subanalytic unit on both $A_1$ and $A_4$.
\end{proof}

Note that the center $\theta_1$ plays a special role in Lemma \ref{lemma:simulPrep}, in the sense that when $y-\theta_1(x)$ and $y-\theta_2(x)$ are equivalent up to multiplication by a strong subanalytic unit, this unit is always constructed to have center $\theta_1$.

\begin{remark}\label{rmk:simulPrep}
In case (iii) of Lemma \ref{lemma:simulPrep} the definition of strong subanalytic units shows that $\frac{\theta_1(x) - \theta_2(x)}{y-\theta_1(x)}$ is bounded on $A$,
and also $\theta_1(x) - \theta_2(x)$ has a constant nonzero sign on the base of $A$, from which it follows that there exists a constant $c$ such that
$|y-\theta_1(x)| > c(\theta_1(x)-\theta_2(x)) > 0$ on $A$.
\end{remark}

%% file: cellPrep1.tex
\section{Cell preparation of subanalytic and constructible functions}\label{s:cellPrep}

This section gives three variants of the subanalytic preparation
theorem. They prepare functions on cells,
rather than on cylinders, in such a specific way that we call them cell preparation theorems. The first two prepare (finite collections of) subanalytic functions, and the third prepares constructible functions.

\begin{definition}\label{def:cell}
A {\bf subanalytic cell} is a set $A\subset\RR^n$ such that
$\Pi_i(A)$ is a subanalytic cylinder for all $i\in\{1,\ldots,n\}$.
There exists a unique increasing function
$\lambda:\{1,\ldots,d\}\to\{1,\ldots,n\}$ whose image is the set
$\{i\in\{1,\ldots,n\} : \text{$\Pi_i(A)$ is fat}\}$.  We call $A$ a
{\bf $\lambda$-cell} if we want to specify $\lambda$.
For any $I \subset \im(\lambda)$, we say that $A$ is {\bf fat in $(x_i)_{i\in I}$}.
\end{definition}

In the previous definition, $A$ is clearly a connected analytic
submanifold of $\RR^n$ of dimension $d$, and the projection
$\Pi_\lambda:A\to\RR^d$
is an analytic isomorphism onto its image, which is an open subanalytic cell in $\RR^d$.

\begin{definition}\label{def:center}
Suppose that $A\subset\RR^n$ is an open subanalytic cell.  For each $i\in\{1,\ldots,n\}$, let $\theta_i:\Pi_{i-1}(A)\to\RR$ be an analytic subanalytic term, and write $\tld{x}_i := x_i - \theta_i(x_1,\ldots,x_{i-1})$.   If $\tld{x}_i\not\in\{-1,0,1\}$ for all $i\in\{1,\ldots,n\}$ and all $x\in A$, then we call $\theta = (\theta_1,\ldots,\theta_n) :A\to\RR^n$ a {\bf center for $A$} and call $\tld{x} = (\tld{x}_1,\ldots,\tld{x}_n)$ the {\bf coordinates for $A$ with center $\theta$}.  If $\theta_1=\cdots=\theta_n=0$, we simply say that ``$0$ is a center for $A$''.

More generally, suppose that $A\subset\RR^n$ is a $d$-dimensional $\lambda$-cell.   A {\bf center for $A$} is a center $\theta$ for $\Pi_\lambda(A)$, and the {\bf coordinates for $A$ with center $\theta$} are the coordinates for $\tld{x}$ for $\Pi_\lambda(A)$ with center $\theta$.
\end{definition}

\begin{notation}
Suppose that $A\subset\RR^n$ is a $d$-dimensional $\lambda$-cell with center $\theta$, and that $\tld{x}$ are the coordinates for $A$ with center $\theta$.
We index $\theta$ and $\tld{x}$ by $\im(\lambda)$ rather than $\{1,\ldots,d\}$, writing $\theta = (\theta_{\lambda(1)},\ldots,\theta_{\lambda(d)})$ and $\tld{x} = (\tld{x}_{\lambda(1)},\ldots,\tld{x}_{\lambda(d)})$, and considering each $\theta_{\lambda(i)}$ to be a function of $(x_{\lambda(1)},\ldots,x_{\lambda(i-1)})$.  For any tuple $\alpha = (\alpha_i)_{i\in\im(\lambda)}$ of rational numbers, let
\[
|\tld{x}|^{\alpha}  = \prod_{i\in\im(\lambda)}|\tld{x}_{i}|^{\alpha_i}.
\]
We shall write $\QQ^{\im(\lambda)}$ for tuples of rational numbers indexed by $\im(\lambda)$, and likewise for $\RR^{\im(\lambda)}$.  If $J\subset\im(\lambda)$ and $\alpha\in\QQ^{\im(\lambda)}$ are such that $\alpha_i = 0$ for all $i\in\im(\lambda)\setminus J$, we say that $\alpha$ {\bf has support in $J$}.
\end{notation}

\begin{definitions}\label{def:Abd}
Let $A\subset\RR^n$ be a $d$-dimensional subanalytic $\lambda$-cell with center $\theta$.  Then for all $i\in\im(\lambda)$, the set $\{\tld{x}_i : x\in A\}$ is contained in either $(-\infty,-1)$, $(-1,0)$, $(0,1)$, or $(1,+\infty)$, so there exist unique $\varepsilon_i,\zeta_i\in\{-1,1\}$ such that $0 < \varepsilon_i \tld{x}_{i}^{\zeta_i} < 1$ for all $x\in A$.  Define
\[
A_{\bd} = \left\{(\varepsilon_i \tld{x}_{i}^{\zeta_i})_{i\in\im(\lambda)} : x\in A\right\}.
\]
Define the isomorphism $F^{[\theta]}_{A}:A\to A_{\bd}$ by $F^{[\theta]}_{A}(x) = (\varepsilon_i \tld{x}_{i}^{\zeta_i})_{i\in\im(\lambda)}$, and define $G^{[\theta]}_{A}:A_{\bd}\to A$ to be the inverse of $F^{[\theta]}_{A}$.  We consider $A_{\bd}$ to be a subset of $\RR^d$ (rather than of $\RR^{\im(\lambda)}$), and write $y = (y_1,\ldots,y_d) = F_{A}^{[\theta]}(x)$ for $x\in A$.  So
\[
y_i = \varepsilon_{\lambda(i)} \tld{x}_{\lambda(i)}^{\,\,\,\zeta_{\lambda(i)}}
\]
for each $i\in\{1,\ldots,d\}$.  Clearly $A_{\bd}$ is an open subanalytic cell in $(0,1)^d$ with center $0$.  Write
\begin{equation}\label{eq:Abd}
\Pi_i(A_{\bd}) = \{y_{\leq i} : y_{<i}\in\Pi_{i-1}(A_{\bd}), a_i(y_{<i}) < y_i < b_i(y_{<i})\}
\end{equation}
for each $i\in\{1,\ldots,d\}$.
\begin{enumerate}{\setlength{\itemsep}{3pt}
\item
A {\bf strong subanalytic unit on $A$ with center $\theta$} is a function of the form $u = U\circ\varphi$, where $\varphi:A\to\RR^N$ is a bounded function of the form
\[
\varphi(x) = (|\tld{x}|^{\beta_1},\ldots,|\tld{x}|^{\beta_N})
\]
for some natural number $N$ and $\beta_1,\ldots,\beta_N\in\QQ^{\im(\lambda)}$, and $U$ is an analytic unit on the closure of the image of $\varphi$.

\item
For $i\in\{1,\ldots,d\}$, call $\tld{x}_{\lambda(i)}$ {\bf asymptotically determined on $A$} if there exists $C > 0$ such that $b_i(y_{<i}) < C a_i(y_{<i})$ on $\Pi_{i-1}(A_{\bd})$.  Otherwise call $\tld{x}_{\lambda(i)}$ {\bf asymptotically undetermined}.

\item
Let $\F$ be a finite family of subanalytic functions on $A$, and consider a set
\[
J \subset \{i \in\im(\lambda): \text{$\tld{x}_i$ is asymptotically undetermined on $A$}\}.
\]
By a joint induction on $d$, we define what it means for $A$ to be {\bf $J$-prepared with center $\theta$} and what it means for $\F$ to be {\bf $J$-prepared on $A$ with center $\theta$}.

If $d = 0$, then $A$ is a singleton and $J = \theta = \emptyset$, and $A$ and $\F$ are called $\emptyset$-prepared with center $\emptyset$.  For $d\geq 1$, say that $A$ is $J$-prepared with center $\theta$ if the collection of functions $\{a_d,b_d,b_d-a_d\}$ is $\{1,\ldots,d-1\}\cap \lambda^{-1}(J)$-prepared on $\Pi_{d-1}(A_{\bd})$ with center $0$ (or center $\emptyset$ when $d=1$), and if the closure of the image of $a_d$ contains $0$.  We say that $\F$ is $J$-prepared on $A$ with center $\theta$ if $A$ is $J$-prepared with center $\theta$, and if for each $f\in\F$ either $f=0$ on $A$ or
\[
f(x) = |\tld{x}|^\alpha u(x)
\]
on $A$ for some strong subanalytic unit $u$ on $A$ with center $\theta$ and some $\alpha\in\QQ^{\im(\lambda)}$ with support in $J$.

When $J = \{i \in\im(\lambda): \text{$\tld{x}_i$ is asymptotically undetermined on $A$}\}$, we simply say that $A$ (or $\F$) is prepared with center $\theta$ (on $A$).

\item
For $i\in\{1,\ldots,d\}$, call $\tld{x}_{\lambda(i)}$ {\bf constrained on $A$} if $a_i > 0$ on $\Pi_{i-1}(A_{\bd})$, and call $\tld{x}_{\lambda(i)}$ {\bf unconstrained on $A$} if $a_i = 0$ on
$\Pi_{i-1}(A_{\bd})$.
}
\end{enumerate}
\end{definitions}


Note that if $A$ is a subanalytic $\lambda$-cell in $\RR^n$ with center $\theta=(\theta_i)_{i\in\im(\lambda)}$, where $n\in\im(\lambda)$, and if $u$ is a strong subanalytic unit on $A$ with center $\theta$, then $A$ is also a subanalytic cylinder in $\RR^n$ with center $\theta_n$ and $u$ is a strong subanalytic unit on $A$ with center $\theta_n$ (in the cylindrical senses of Definitions \ref{def:subPrep}).


\begin{lemma}\label{lem:asympt}
Let $A\subset(0,1)^n$ be an open cell $A\subset(0,1)^n$ which is prepared with center $0$, and let
$
J= \{i : \tld{x}_i \mbox{  is asymptotically undetermined on $A$  }\}.
$
Then for any nonzero $\gamma\in\QQ^n$ with support in $J$, the image of the function $x\in A\mapsto x^\gamma$ is not contained in a compact subset of $(0,+\infty)$.
\end{lemma}

\begin{proof}
The proof is by induction on $n$.  For $n=1$ this is clear. Suppose the statement is known for $n-1$.  Write
$$
A=\{x : x_{<n}\in\Pi_{n-1}(A),\  a(x_{<n}) <   x_n < b(x_{<n}) \},
$$
where either $a = 0$ or $a(x_{<n}) = x_{<n}^{\alpha} u(x)$, and $b(x_{<n}) = x_{<n}^{\beta}v(x)$, for some
$\alpha = (\alpha_1,\ldots,\alpha_{n-1})$ and $\beta = (\beta_1,\ldots,\beta_{n-1})$ in $\QQ^{n-1}$ with support in $J\cap\{1,\ldots,n-1\}$.
We may suppose that $n\in J$ and that $\gamma_n\not = 0$, since otherwise we are done.
If $a = 0$, then by fixing $x_{<n}\in\Pi_{n-1}(A)$ and letting $x_n\to 0$, $x^\gamma$ tends to either $0$ or $+\infty$, and we are done.  So assume that $a > 0$.  Since $n\in J$, it follows that $\alpha\not =  \beta$.  Define, on $\Pi_{n-1}(A)$, the functions
$$
h_1(x_{<n }):=x_{<n }^{\gamma_{<n}}a(x_{<n })^{\gamma_n} \ \mbox{ and } \  h_2(x_{<n }):=x_{<n }^{\gamma_{<n}}b(x_{<n })^{\gamma_n}.
$$
Clearly $g_\gamma$ takes all values between $h_1(x_{<n })$ and $h_2(x_{<n })$ for any $x_{<n }\in \Pi_{n-1}(A)$. Since $\alpha\not =  \beta$,
at least one of the tuples $\gamma_{<n}+\gamma_n\alpha$ or $\gamma_{<n}+\gamma_n\beta$ is nonzero.
Hence, by induction, at least one of the $h_i$ has an image which is not contained in a compact subset of $(0,+\infty)$, and so does $g_\gamma$.
\end{proof}

The following are immediate consequences of Lemma \ref{lem:asympt}.

\begin{corollaries}\label{cor:asympt}
Using the notation from Definitions \ref{def:Abd}, suppose that $A$ is prepared with center $\theta$, and let $J = \{i \in\im(\lambda): \text{$\tld{x}_i$ is asymptotically undetermined on $A$}\}$.  For the functions $a_i$ and $b_i$ given in \eqref{eq:Abd}, write $a_i(y_{<i}) = y_{<i}^{\alpha_i} u_i(y_{\leq i})$ provided that $a_i > 0$ on $\Pi_{i-1}(A_{\bd})$, and write $b_i(y_{\leq i}) = y_{<i}^{\beta_i}v_i(y_{\leq i})$, where $\alpha_i,\beta_i\in\QQ^{i-1}$ have support in $\{1,\ldots,i-1\}\cap\lambda^{-1}(J)$, and $u_i$ and $v_i$ are strong subanalytic units on $\Pi_i(A_{\bd})$ with center $0$.
\begin{enumerate}{\setlength{\itemsep}{3pt}
\item
For each $i\in\{1,\ldots,d\}$, $\tld{x}_{\lambda(i)}$ is asymptotically determined on $A$ if and only if $a_i > 0$ and $\alpha_i = \beta_i$, if and only if $a_i/b_i$ is a strong subanalytic unit.

\item
The prepared form of a subanalytic function $f:A\to\RR$ is unique in the following sense: if $f(x) = |\tld{x}|^{\alpha} u(x) = |\tld{x}|^{\beta} v(x)$ on $A$, where $\alpha,\beta\in\QQ^n$ have support in $J$ and $u$ and $v$ are strong subanalytic units on $A$ with center $\theta$, then $\alpha = \beta$.
}
\end{enumerate}
\end{corollaries}

\begin{definition}
Inductively define a {\bf subanalytic cell decomposition} of a subanalytic
set $X\subset\RR^n$ to be a finite partition $\A$ of $X$ by
subanalytic cells such that, when $n > 0$, $\{\Pi_{n-1}(A) :
A\in\A\}$ is a subanalytic cell decomposition of $\Pi_{n-1}(X)$.
\end{definition}

\begin{definition}
Let $\F$ be a finite collection of real-valued subanalytic functions on a subanalytic set $X$.  A {\bf cell preparation} of $\F$ is a finite subanalytic cell decomposition of $X$, say $\A$, such that for each $A\in\A$ there exists a center $\theta$ for $A$ such that $\F$ is prepared on $A$ with center $\theta$.  We call $\theta$ the center {\bf associated} with $A$ by the preparation.
\end{definition}

\begin{theorem}\label{thm:cellPrep}
Suppose that $\F$ is a finite collection of real-valued subanalytic
functions on a subanalytic set $X$ and that $\S$ is a finite set of
subanalytic subsets of $X$.  There exists a cell preparation of
$\F$ compatible with $\S$.
\end{theorem}
\begin{proof}
Apply Theorem \ref{thm:subPrep} to $\F$ and $\S$, and let $\A$ be
the finite partition of $X$ into subanalytic cylinders, compatible
with $\S$, which is given by the preparation. The proof now proceeds
by induction on $n$, where $X\subset\RR^n$.
 The case that $n=0$, with $\RR^{0}=\{0\}$, being trivial, suppose that $n \geq 1$.  By further partitioning the cylinders in
$\A$, we may assume that for each fat cylinder $A\in\A$, with
associated center $\theta_n:\Pi_{n-1}(A)\to\RR$,
\begin{equation}\label{eq:Acylinder}
A = \{x : x_{<n}\in\Pi_{n-1}(A), a(x_{<n}) < \varepsilon_i \tld{x}_{n}^{\zeta_n} < b(x_{<n})\}
\end{equation}
for some $\varepsilon_n,\zeta_n\in\{-1,1\}$ and subanalytic functions $a,b:\Pi_{n-1}(A)\to[0,1]$ such that $a < b$ and either $a = 0$ or $a>0$, where $\tld{x}_n = x_n - \theta_n(x_{<n})$.
%
%
%
%
By partitioning the cylinders in $\A$ even further, we may also
assume that $\A$ is a cylindrical decomposition of $X$ over
$\Pi_{n-1}(X)$, namely, that there exists a finite partition $\B$ of
$\Pi_{n-1}(X)$ into subanalytic sets such that
$\A=\bigcup_{B\in\B}\A_B$, where $\A_B$ is a set of disjoint
cylinders with base $B$.  Fix $B\in\B$. For each fat cylinder $A\in
\A_B$ there exists a center $\theta_n:B\to\RR$ such that each
$f\in\F$ can be written in the form
\begin{equation}\label{eq:prepf}
f(x) = c(x_{<n}) |\tld{x}_n|^r
u(a_1(x_{<n})|\tld{x}_n|^{r_1},\ldots,a_N(x_{<n})|\tld{x}_n|^{r_n})
\end{equation}
on $A$.  Now apply the
induction hypothesis to $\B$ and the following set of functions on each $B\in \B$ (extended by zero outside $B$):
\begin{itemize}
\item
For each $f\in\F$ and each fat cylinder $A\in\A_B$, the functions
$c, a_1,\ldots,a_N$ from \eqref{eq:prepf}.


\item
For each fat cylinder $A\in\A_B$, the functions $a$, $b$, and $b-a$ from \eqref{eq:Acylinder}.

\item
For each thin cylinder $A\in\A_B$, say of the form $A = \{x\in
B\times\RR : x_n = d(x_{<n})\}$, and each $f\in\F$, the function
$f\circ d$.
\end{itemize}
This gives a decomposition of $X$ into subanalytic cells which associates to each cell $A$ in the decomposition a center $\theta = (\theta_1,\ldots,\theta_n)$ such that $\Pi_{n-1}(A)$ is prepared with center $\theta_{<n}$, and each $f\in\F$ is either identically zero on $A$ or can be written in the form
\[
f(x) = |\tld{x}|^\alpha u(x)
\]
on $A$ for a strong subanalytic unit $u(x)$ with center $\theta$, where for each $i < n$, $\alpha_i = 0$ if $\tld{x}_i$ is asymptotically determined on $A$.  If $A$ is thin in $x_n$, then $\F$ is prepared on $A$ with center $\theta$.  If $A$ is fat in $x_n$ and
$a = 0$,
then we are also done on $A$. So let $A$ be
fat in $x_n$ with $a > 0$.
There are two possible reasons why $\F$ might not be prepared on $A$ with center $\theta$:
\begin{enumerate}{\setlength{\itemsep}{3pt}
\item
It might be that $\alpha_n \neq 0$ but $\tld{x}_n$ is asymptotically determined on $A$, that is,
the image of $a/b$ is contained in a compact subset of $(0,+\infty)$.

\item
It might be that
$a$ and $b$ are both strong subanalytic units on $\Pi_{n-1}(A)$ with center $\theta_{<n}$,
which violates the requirement that $0$ be in the closure of the image of $a$.
}
\end{enumerate}

Suppose that we are in case 1. Then
\[
f(x) = |\tld{x}_{<n}|^{\alpha_{<n}}|a(\tld{x}_{<n})|^{\alpha_n}\cdot\left[
\frac{ | \tld{x}_n|^{\alpha_n}}{|a(\tld{x}_{<n})|^{\alpha_n}} u(x)\right].
\]
The expression in square brackets is a strong subanalytic unit, and we are done.

Finally we treat case 2, where we suppose that $\varepsilon_n = \zeta_n = 1$, the other cases being similar. Since clearly $\tld{x}_n$ is asymptotically determined
on $A$, our treatment of case 1 shows that we can assume that $\alpha_n=0$. We will change the center so that any strong subanalytic unit $u$ with the old center remains a strong subanalytic unit with the new center.
Write
\[
u(x) = U(|\tld{x}|^{\gamma_1},\ldots,|\tld{x}|^{\gamma_k})
\]
for some natural number $k$, tuples $\gamma_1,\ldots,\gamma_k\in\QQ^n$,
and a function $U$ which is analytic on the closure of the image
of $x\in A\mapsto(|\tld{x}|^{\gamma_1},\ldots,|\tld{x}|^{\gamma_k})$,
which is bounded in $\RR^k$. For any $i$,
the closure of the image of the map $x\in A\mapsto
(|\tld{x}_{<n}|^{\gamma_{<n}(i)}, \tld{x}_n)$ is a
compact subset of $\RR\times(0,+\infty)$, and the map $(s,t)\mapsto s\cdot
t^{\gamma_n(i)}$ is analytic on this set. So we may assume that $u(x)$ is of the
form
\[
u(x) =
U(|\tld{x}_{<n}|^{\gamma_1},\ldots,|\tld{x}_{<n}|^{\gamma_{\ell-1}},\tld{x}_n)
\]
for an analytic function $U(t_1,\ldots,t_{\ell})$ and
$\gamma_1,\ldots,\gamma_{\ell-1}\in\QQ^{n-1}$, for some natural number
$\ell$.  But then we may write $A$ as
\[
A = \{x\in\Pi_{n-1}(A)\times\RR : 0 < x_n - (\theta_n + a)(x_{<n}) <
(b-a)(x_{<n})\}
\]
and write $u$ as
\[
u(x) = U(\tld{x}_{<n}^{\gamma_1},\ldots,\tld{x}_{<n}^{\gamma_{\ell-1}},(x_n
- (\theta_n+a)(x_{<n})) + a(x_{<n})).
\]
The function $b-a$ is prepared on $\Pi_{n-1}(A)$ with center $\theta_{<n}$, so $A$ is prepared with center $(\theta_{<n},\theta_n+a)$.  Also, the function $a$ is a strong subanalytic unit on $\Pi_{n-1}(A)$ with center $\theta_{<n}$, so $u(x)$ is a strong subanalytic unit on $A$ with center $(\theta_{<n},\theta_n+a)$.  Thus $\F$ is prepared on $A$ with center $(\theta_{<n},\theta_n+a)$.
\end{proof}

In order to apply the cell preparation theorem multiple times we
will need the following variant of Theorem \ref{thm:cellPrep}.

\begin{theorem}\label{thm:iteratedCellPrep}
Suppose that $\F$ is a finite set of real-valued subanalytic
functions on a subanalytic set $X$, that $\S$ is a subanalytic cell
decomposition of $X$, and that for each $S\in\S$ we have an
associated center $\theta_S$ for $S$.
There exists a cell preparation of $\F$ compatible with $\S$
such that for each cell $A$ in this preparation, say a
$\lambda$-cell with center $\theta =
(\theta_{\lambda(1)},\ldots,\theta_{\lambda(d)})$, if $S$ is the
unique cell in $\S$ containing $A$,
then for each $i\in\im(\lambda_S)$
exactly one of the following holds:
\begin{renumerate}{\setlength{\itemsep}{3pt}
\item
The cell $A$ is thin in $x_i$, and
\[
x_i - \theta_{S,i}(x_1,\ldots,x_{i-1}) = a(x_1,\ldots,x_{i-1})
\]
on $A$ for a subanalytic term $a(x_1,\ldots,x_{i-1})$ which is
prepared on $\Pi_{i-1}(A)$ with center
$(\theta_j)_{j\in\im(\lambda),j<i}$.

\item
The cell $A$ is fat in $x_i$, and
\[
x_i - \theta_{S,i}(x_1,\ldots,x_{i-1}) = a(x_1,\ldots,x_{i-1})
u(x_1,\ldots,x_i)
\]
on $A$ for a subanalytic term $a(x_1,\ldots,x_{i-1})$ which is
prepared on $\Pi_{i-1}(A)$ with center
$(\theta_j)_{j\in\im(\lambda),j<i}$ and a strong subanalytic unit
$u(x_1,\ldots,x_i)$ on $\Pi_i(A)$ with center
$(\theta_j)_{j\in\im(\lambda),j\leq i}$.

\item
The cell $A$ is fat in $x_i$,
$\theta_i = \theta_{S,i}$, and
$x_i - \theta_i(x_1,\ldots,x_{i-1})$
is asymptotically undetermined on $A$.
}
\end{renumerate}
\end{theorem}

\begin{proof}
The proof of Theorem \ref{thm:iteratedCellPrep} is a simple
modification of the proof of Theorem \ref{thm:cellPrep} using Lemma
\ref{lemma:simulPrep}.  Indeed, for the cases  $n\geq 1$,
begin as before by applying Theorem \ref{thm:subPrep} to get a
cylindrical preparation of $\F$ compatible with $\S$.  Call it $\A$
as before.
Consider a fat cylinder $A\in\A$, and fix the unique $S\in\S$ such that $A\subset S$. Let $\theta_{A,n}$ be the center for $A$ given by the cylindrical preparation $\A$.  Thus each $f\in\F$ can be written in the form
\begin{equation}\label{eq:fPrepOnA}
f(x) = a(x_{<n})|x_n - \theta_{A,n}(x_{<n})|^r u(x_{<n},x_n - \theta_{A,n}(x_{<n}))
\end{equation}
on $A$, where $u(x_{<n},x_n)$ is a strong subanalytic unit with center $0$ on \[
\{(x_{<n},x_n-\theta_{A,n}(x_{<n})) : x\in A\}.
\]
By applying Lemma \ref{lemma:simulPrep} to $\theta_{S,n}$ and $\theta_{A,n}$ on $A$, we get a partition $\A'_A$ of $A$ by subanalytic cylinders such that for each fat cylinder $A'\in\A'_A$, either $\theta_{A,n} = \theta_{S,n}$ on $A'$, or at least one of the following holds on $A'$,
\begin{enumerate}
\item
$\displaystyle x_n - \theta_{S,n}(x_{<n}) = (\theta_{A,n}(x_{<n}) -
\theta_{S,n}(x_{<n})) \cdot \left[1 + \frac{x_n -
\theta_{A,n}(x_{<n})}{\theta_{A,n}(x_{<n}) - \theta_{S,n}(x_{<n})} \right]$,

\item
$\displaystyle x_n - \theta_{A,n}(x_{<n}) = (\theta_{S,n}(x_{<n}) -
\theta_{A,n}(x_{<n})) \cdot \left[1 + \frac{x_n -
\theta_{S,n}(x_{<n})}{\theta_{S,n}(x_{<n}) - \theta_{A,n}(x_{<n})} \right]$,


\item
$\displaystyle x_n - \theta_{A,n}(x_{<n}) = (x_n - \theta_{S,n}(x_{<n}))
\cdot \left[1 + \frac{\theta_{S,n}(x_{<n}) - \theta_{A,n}(x_{<n})}{x_n -
\theta_{S,n}(x_{<n})} \right]$,
\end{enumerate}
where the expressions in square brackets are strong subanalytic units.  We now assign a center $\theta_n$ to $A'$ as follows: In case 1 we let $\theta_n = \theta_{A,n}$.  In cases 2 and 3 we let $\theta_n = \theta_{S,n}$, and for each $f\in\F$ we use the equation from either 2 or 3 to express each instance of $x_n - \theta_{A,n}(x_{<n})$ in \eqref{eq:fPrepOnA} in terms of $x_n - \theta_n(x_{<n})$ so that $f$ is prepared on $A'$ with center $\theta_n$.

Doing this for each fat cell $A$ in $\A$ gives a cylindrical preparation $\A' = \bigcup_{A\in\A}\A'_A$ of $\F$ on $X$.  We now proceed as before in the proof of Theorem \ref{thm:cellPrep}, except that we use $\A'$ in place of $\A$, we include the functions $\theta_{A,n} - \theta_{S,n}$ (extended by zero), and on thin cylinders $A$ the functions $x_n- \theta_{S,n}$ (extended by zero), in the set of functions to which the induction hypothesis is applied, and we strengthen our induction hypothesis to  assume that cases (i), (ii), or (iii) of the Theorem hold in the variables $x_1,\ldots,x_{n-1}$.
\end{proof}


\begin{theorem}[Cell preparation of constructible functions]\label{thm:constrPrep}
Suppose that $\F$ is a finite set of functions in $\C(X)$ for a subanalytic set $X\subset\RR^n$, and that $\S$ is a finite set of subanalytic subsets of $X$.  There exists a finite subanalytic cell decomposition of $X$, compatible with $\S$, such that for each cell $A$ in this decomposition, there exists a center $\theta$ such that  $A$ is prepared with center $\theta$ and the following hold:

Suppose that $A$ is a $d$-dimensional $\lambda$-cell, let $\tld{x} =
(\tld{x}_{\lambda(1)},\ldots,\tld{x}_{\lambda(d)})$ be the
coordinates for $A$ with center $\theta$, and let
\[
J = \{j\in\im(\lambda) : \text{$\tld{x}_j$ is
asymptotically undetermined on $A$}\}.
\]
Each $f\in\F$ can be written as a finite sum
\begin{equation}\label{eq:sumT}
f(x) = \sum_{i\in I} T_i(x)
\end{equation}
on $A$, where each term $T_i(x)$ is of the form
\begin{equation}\label{eq:formT}
u_i(\tld{x})
\left(\prod_{j\in J} |\tld{x}_j|^{\alpha_j(i)}
    (\log|\tld{x}_j|)^{\ell_j(i)}\right)
\end{equation}
for rational numbers $\alpha_j(i)$, integers $\ell_j(i)\geq 0$,
and strong subanalytic units $u_i$ with center $\theta$.  Moreover, the tuples $(\ell_j(i))_{j\in J}$ are distinct for different $i$ in $I$.
\end{theorem}

\begin{proof}
Each $f\in\F$ may be written in the form
\[
f(x) = \sum_{i} g_i(x) \prod_{k}\log g_{ik}(x),
\]
where the $g_i:X\to\RR$ and
$g_{ik}:X\to(0,+\infty)$ are subanalytic and $i$ and
$k$ run over finite index sets.  Apply Theorem \ref{thm:cellPrep} to
all the $g_{ik}$ for all $f\in\F$.
Fix a cell $A'$ from this preparation, with associated center $\theta'$.  Let $x'$ be the coordinates for $A'$ with center $\theta'$,
%
%
and let $J'$ be the set of indices $j$ in $\{1,\ldots,n\}$ for which
$x_j$ is asymptotically undetermined on $A'$.  Thus we may write each
$f\in\F$ in the form
\[
f(x) = \sum_{i} g_i(x)\prod_{k} \log\left( \left(\prod_{j\in
J'}|x'_j|^{\gamma_j(i,k)}\right) u_{i,k}(x')
\right)
\]
on $A'$ for rational numbers $\gamma_j(i,k)$ and strong subanalytic
units $u_{i,k}$.  By expanding the logarithmic expressions and by using the
facts that logarithms of strong subanalytic units are subanalytic and
that sums of subanalytic functions are subanalytic, we may write each $f\in\F$
in the form
\begin{equation}\label{eq:logPrep1}
f(x) = \sum_{i\in I_f} f_i(x)
\left(\prod_{j\in J'}(\log|x'_j|)^{\ell_{f,j}(i)}\right)
\end{equation}
on $A'$, where the $(\ell_{f,j}(i))_{j\in J'}$ are
tuples of natural numbers which are distinct for
different $i$ in the
finite index set $I_f$, the $f_i$ are subanalytic functions.

Consider the following ``lexicographical'' ordering of the power set
of $\{1,\ldots,n\}$:  for any subsets $M$ and $N$ of $\{1,\ldots,n\}$,
define $M < N$ if and only if there exists $m\in\{1,\ldots,n\}$ such that
$M\cap\{m+1,\ldots,n\} = N\cap\{m+1,\ldots,n\}$ and $\max(
M\cap\{1,\ldots,m\}) < \max(N\cap\{1,\ldots,m\})$, with the
understanding that $\max(\emptyset)$ is defined to be $0$.  Note
that the smallest member of this ordering is the empty set.  With
respect to this ordering, we induct on the set
\begin{equation}\label{eq:L'}
L' := \{j\in J' : \text{$\ell_{f,j}(i) > 0$ for some
$f\in\F$ and $i\in I_f$}\}.
\end{equation}

If $L'$ is empty, then each $f\in\F$ is subanalytic on $A'$, in
which case we are done by applying Theorem \ref{thm:cellPrep} to
$\F$ on $A'$.  So assume that $L'$ is nonempty. Apply
Theorem \ref{thm:iteratedCellPrep} to $\{f_i : f\in\F, i\in I_f\}$. Fix an open
cell $A''\subset A'$ from this preparation, with associated center
$\theta''$.  Let $x''$ be the coordinates for $A''$ with center $\theta''$, and let $J''$ be the set of indices $j$ in $\{1,\ldots,n\}$ for which $x''_j$ is
asymptotically undetermined on $A''$. Thus each $f\in\F$ can be written as
\begin{equation}\label{eq:logPrep2}
f(x) = \sum_{i\in I_f} u_{f,i}(x'') \left(\prod_{j\in J''}
|x''_j|^{\alpha_{f,j}(i)}\right)
\left(\prod_{j\in J'}(\log|x'_j|)^{\ell_{f,j}(i)}\right)
\end{equation}
on $A''$ for some rational numbers $\alpha_{f,j}(i)$ and strong
subanalytic units $u_{f,i}$, and each $x'_j$  can be expressed in terms of $x''$ in the form
(i), (ii), or (iii) from Theorem \ref{thm:iteratedCellPrep}.
We are done if case (iii) holds for every $j$ in $L'$, so assume otherwise.  Let $m$ be the maximum $j$ in $L'$ such that $x'_j$ is of the form (i) or (ii).  By expanding each of the logarithmic terms $(\log (x'_m))^{\ell_{f,m}(i)}$ in \eqref{eq:logPrep2}, where $x'_m$ is expressed in terms of $x''_{\leq m}$ in accordance with the form (i) or (ii), we may write each $f\in\F$ in the form \eqref{eq:logPrep1} on $A''$ but with a new set \eqref{eq:L'}, which we call $L''$, which is lexicographically less than $L'$, since $L''\cap\{m+1,\ldots,n\} = L'\cap\{m+1,\ldots,n\}$ and $\max(L''\cap\{1,\ldots,m\}) < m = \max(L'\cap\{1,\ldots,m\})$.  We are done by the induction hypothesis.
\end{proof}

\begin{remark}\label{rmk:unbddFiberCell}
Remark \ref{rmk:unbddFiberCylinder} and the proof of Theorem \ref{thm:cellPrep} imply that the cell preparation constructed by Theorem \ref{thm:cellPrep} has the following property:
\begin{quote}
For every cell $A$ in the preparation with associated center $\theta = (\theta_1,\ldots,\theta_n)$, and every $i\in\{1,\ldots,n\}$, if each of the fibers of $\Pi_i(A)$ over $\Pi_{i-1}(A)$ are unbounded, then $\theta_i = 0$.
\end{quote}
Similarly, if in Theorem \ref{thm:iteratedCellPrep} we assume that the given cell decomposition $\S$, with associated centers $\theta_S$, has this property, then the cell preparation constructed by Theorem \ref{thm:iteratedCellPrep} has this property.  Therefore the cell preparation constructed in Theorem \ref{thm:constrPrep} also has this property, because its proof applies Theorems \ref{thm:cellPrep} and Theorem \ref{thm:iteratedCellPrep} in succession.
\end{remark}

%% file: sliver1.tex
%

\section{Sliver functions}\label{s:sliver}

Slivers will be used to study relative asymptotic classes of basic terms in the proof of Proposition \ref{prop:constrPrep}.

\begin{definition}
\label{def:sliver}
Let $\epsilon\in(0,1)$, and for each $i\in\{2,\ldots,n\}$ let $(p_i,q_i)$ be an interval with $0 < p_i < q_i < +\infty$.  Write $(p,q) = (p_2,q_2)\times \cdots \times (p_n,q_n)$.  Define $\psi:(0,\epsilon)\times(p,q)\to\RR^n$ by
\[
\psi(t) = (t_1, t_{1}^{t_2}, t_{1}^{t_3},\ldots,t_{1}^{t_n}),
\]
where $t = (t_1,\ldots,t_n)$.   We call $\psi$ a {\bf sliver function for} a set $A\subset \RR^n$ if $\im(\psi)\subset A$, and in this case we also call $\im(\psi)$ a {\bf sliver in $A$}.
\end{definition}

\begin{remark}\label{rmk:sliver}
A sliver function $\psi:(0,\epsilon)\times(p,q)\to A$ is an analytic isomorphism onto its image, since it is clearly an injective analytic map and its Jacobian matrix is lower triangular with diagonal entries which are nonzero on $(0,\epsilon)\times(p,q)$. In particular, the image of $\psi$ is an open subset of $A$.
\end{remark}

If $\psi:(0,\epsilon)\times(p,q) \to \RR^n$ is a sliver function and $\alpha = (\alpha_1,\ldots,\alpha_n)\in\QQ^n$, then
\[
\psi(t)^\alpha = t_{1}^{\alpha_1+\sum_{j=2}^{n}\alpha_j t_j},
\]
so the asymptotic behavior of $\psi(t)^\alpha$ as $t_1\to 0$ is determined by the values of the affine function $t_{>1}\mapsto \alpha_1+\sum_{j=2}^{n}\alpha_j t_j$ on $(p,q)$.  By shrinking $(p,q)$ we may gain more control over the asymptotics of $\psi(t)^\alpha$ as $t_1\to 0$.
The following lemma is therefore useful when studying the relative asymptotics of finite sets of rational monomial functions along slivers, and through extension via the preparation theorems, of prepared subanalytic and constructible functions along slivers.

%
%

\begin{lemma}\label{lemma:affine}
Let $f_1,\ldots,f_k$ be distinct affine functions on an open set $U\subset\RR^n$.  Then there exist $c > 0$, $i\in\{1,\ldots,k\}$, and an open set $V$ in $U$ such that $f_i(x) + c < f_j(x)$ on $V$ for all $j\neq i$.
\end{lemma}

\begin{proof}
The set
$
C = \{x\in U : \text{$f_i(x) = f_j(x)$ for distinct $i,j \in \{1,\ldots,k\}$}\}
$
is closed and nowhere dense in $U$, so $U\setminus C$ is open and dense in $U$.  Choose $a\in U\setminus C$, and note that the values of $f_1(a),\ldots,f_k(a)$ are distinct.  Let $V$ be any sufficiently small neighborhood of $a$ in $U\setminus C$.
\end{proof}

\begin{lemma}\label{lemma:sliverMonomialLimit}
Fix a sliver function $\psi:(0,\epsilon)\times(p,q)\to \RR^n$ and fix $\beta = (\beta_1,\ldots,\beta_n) \in\QQ^n\setminus\{0\}$  such that $\psi(t)^\beta$ is bounded on $(0,\epsilon)\times(p,q)$. Then, up to shrinking $(p,q)$, we can ensure that $\psi(t)^{\beta}\to 0$ as $t_1\to 0$ uniformly on $(p,q)$.
\end{lemma}
\begin{proof}
The function
 $
\psi(t)^\beta = t_{1}^{\beta_1+\sum_{i=2}^{n}\beta_i t_i}
 $
is bounded on $(0,\epsilon)\times(p,q)$, thus $\beta_1+\sum_{i=2}^{n}\beta_i t_i \geq 0$ on $(p,q)$.  Since $\beta\not=0$, up to shrinking the intervals $(p_i,q_i)$, we can ensure that $\beta_1+\sum_{i=2}^{n}\beta_i t_i \geq c$ on $(p,q)$ for some $c > 0 $.  Then $0 < \psi(t)^\beta \leq t_{1}^{c}$, so $\psi(t)^\beta$  tends uniformly to $0$ as $t_1\to 0$.
\end{proof}

\begin{proposition}\label{lemma:sliver}
Let $A\subset(0,1)^n$ be an open, subanalytic cell
which is prepared with center $0$ and is such that $x_1,\ldots,x_n$ are all asymptotically undetermined on $A$. Then there exists a sliver function for $A$.
\end{proposition}

\begin{proof}
We proceed by induction on $n$. If $n=1$ the statement is easy, since $A$ is then of the form $(0,b)$ with $b \leq 1$. Suppose next that $n>1$ and that we have chosen a sliver function $\psi':  (0,\epsilon)\times(p',q')\to \Pi_{n-1}(A)$ for $\Pi_{n-1}(A)$, where $(p',q') = (p_2,q_2)\times\cdots\times(p_{n-1},q_{n-1})$.
Write
$$
A=\{x : x_{<n}\in\Pi_{n-1}(A),\  a(x_{<n})<   x_n < b(x_{<n}) \}.
$$
We suppose that $a>0$ on $\Pi_{n-1}(A)$, for the case $a = 0$ is proven by simply omitting $a$ in the following argument.  Hence we can write
$$
a(x_{<n}) = u(x_{<n}) x_1^{\alpha_1} \cdots x_{n-1}^{\alpha_{n-1}}
$$
and
$$
b(x_{<n}) = v(x_{<n}) x_1^{\beta_1} \cdots x_{n-1}^{\beta_{n-1}} ,
$$
for unique $\alpha = (\alpha_1,\ldots,\alpha_n)$ and $\beta = (\beta_1,\ldots,\beta_n)$ in $\QQ^{n-1}$.  Consider the inequalities
\begin{equation}\label{eq:sliver}
a(\psi'(t_{<n})) = u(\psi'(t_{<n})) t_{1}^{\alpha_{1} + \sum_{j=2}^{n-1} \alpha_{j} t_j}
 <
 t_{1}^{t_n}
 <
 v (\psi'(t_{<n})) t_{1}^{\beta_{ 1} + \sum_{j=2}^{n-1} \beta_{j} t_j} =  b(\psi_0(t_{<n})).
\end{equation}
on $(0,\epsilon)\times(p',q')$.  Note that $0 < a < b$ on $\Pi_{n-1}(A)$, so for each $t_{<n}$ there exists $t_n$ satisfying the above inequalities.  Taking the logarithm with base $t_1$ of these inequalities, one gets
\begin{equation}\label{eq:logSliver}
\log_{t_1}(u(\psi_0(t_{<n}))) + \alpha_{1} + \sum_{j=2}^{n-1} \alpha_{j} t_j
>
t_n
>
\log_{t_1}(v(\psi_0(t_{<n}))) +\beta_{ 1} + \sum_{j=2}^{n-1} \beta_{j} t_j.
\end{equation}
Note that $\log_{t_1}(u(\psi_0(t_{<n})))$ and $\log_{t_1}(v(\psi_0(t_{<n})))$ go to zero uniformly when $t_1\to 0$,  so $\alpha_{1} + \sum_{j=2}^{n-1} \alpha_{j} t_j \geq \beta_{ 1} + \sum_{j=2}^{n-1} \beta_{j} t_j$ on $(p',q')$.  Since $x_n$ is asymptotically undetermined, one has $\alpha\not=\beta$, so up to shrinking $(p',q')$, we may fix positive constants $p_n < q_n$ and $\delta$ such that
\[
\alpha_{1} + \sum_{j=2}^{n-1} \alpha_{j} t_j \geq q_n +\delta \quad\text{and}\quad p_n - \delta \geq \beta_{ 1} + \sum_{j=2}^{n-1} \beta_{j} t_j
\]
on $(p',q')$.  Put $(p,q) = (p_2,q_2)\times\cdots\times(p_n,q_n)$.   Up to shrinking $\epsilon$, we may ensure that $|\log_{t_1}(u(\psi_0(t_{<n})))|$ and $|\log_{t_1}(v(\psi_0(t_{<n})))|$ are bounded above by $\delta$ on $(0,\epsilon)\times(p,q)$.  Thus \eqref{eq:logSliver} holds on $(0,\epsilon)\times(p,q)$, so \eqref{eq:sliver} does as well.
\end{proof}

%% file: goodcoor1.tex
%
%

\section{Coordinate transforms}\label{sec:goodCoord}

Let $A\subset(0,1)^n$ be an open subanalytic cell which is prepared with center $0$.
In order to use the existence of slivers, we will perform a coordinate change $H:B\to A$ such that the transformed cell $B$ satisfies the conditions of Proposition \ref{lemma:sliver}. Let
\[
J = \{j\in\{1,\ldots,n\} : \text{$x_j$ is asymptotically undetermined on $A$}\},
\]
and write
$
\{1,\ldots,n\}\setminus J = \{d_1 < \cdots < d_k\}$ for some $k\geq 0$.
Define $H := H_{d_1}\circ\cdots\circ H_{d_k}:B\to A$, where the $H_{d_i}$ are defined as follows.

Writing $d=d_1$, one has  $d > 1$, as follows from Definitions \ref{def:Abd}. By
Definitions \ref{def:Abd} and Corollary \ref{cor:asympt}.1, we can write
\[
\Pi_d(A ) = \{x_{\leq d} : x_{<d}\in\Pi_{d-1}(A ),\ x_{<d}^{\alpha} u(x_{<d}) < x_d < x_{<d}^{\alpha} v(x_{<d})\}
\]
for some $\alpha\in\QQ^{d-1}$ with support in $\{1,\ldots,d-1\}\cap J$ and strong subanalytic units $u$ and $v$ such that $x_{<d}^{\alpha} v(x_{<d}) - x_{<d}^{\alpha} u(x_{<d})$ is $J$-prepared on $\Pi_{d-1}(A )$.  It follows that $v-u$ is also $J$-prepared on $\Pi_{d-1}(A )$. Let $R_d>0$ be a constant such that $v-u\leq R_d$ on $\Pi_{d-1}(A )$. Define $H_d:\Pi_{d-1}( A  )\times(0,1)^{n-d+1}\to\RR^n$ by
\[
H_d(y) = (y_1,\ldots,y_{d-1},\,\,\, y_{<d}^{\alpha}(R_d y_d + u(y_{<d})),\,\,\, y_{d+1},\ldots,y_n),
\]
and define
\[
D := H_{d}^{-1}(A ).
\]
By restricting to $D$ this defines a map $H_d:D\to A $ where we let $y$ range over $D$ and write $x = H_d(y)$.
Clearly $D$ is an open cell contained in $(0,1)^n$,
\[
\Pi_d(D) = \left\{y_{\leq d} : y_{<d}\in\Pi_{d-1}(D),
0 < y_d<  \frac{v(y_{<d})-u(y_{<d})}{R_d}\right\},
\]
and thus $y_d$ is asymptotically undetermined on $D$.  Note also that
$R_d y_d + u(y_{<d})$
is a strong subanalytic unit on $\Pi_d(D)$ with center $0$.
Now construct $H_{d_2}$ as $H_d$, but by starting with $D$ instead of with $A $, and so on, up to the map $H_{d_k}$.

\begin{lemma}\label{lem:goodCoord}
With notation from the above construction, and writing $x = H(y)$
for  $y = (y_1,\ldots,y_n)$ in $B$, the set $B$ is an open, subanalytic cell in $(0,1)^n$ which is $J$-prepared with center $0$, and for all $i$,
\begin{renumerate}

\item $y_i$ is asymptotically undetermined on $B$,

\item $y_i$ is constrained on $B$ if and only if $x_i$ is constrained on $A$ and $i \in J$.
\end{renumerate}
Moreover, for any subanalytic function $g:A\to\RR$ which is prepared on $A$ with center $0$, $g\circ H$ is $J$-prepared on $B$ with center $0$, and for every strong subanalytic unit $w$ on $A$ with center $0$, $w\circ H$ is a strong subanalytic unit on $B$ with center $0$.
\end{lemma}

\begin{proof}
The last sentence follows with a similar argument as in case 2 of the proof of Theorem \ref{thm:cellPrep}, since there a similar transformation is performed. The other statements follow from the construction.
\end{proof}

%% file: intprep1.tex
\section{Integrability and preparation of constructible functions}\label{s:constrPrep}

In this section we prove our
key technical result, Proposition \ref{prop:constrPrep},
which gives a strong connection between  integrability conditions and  the  preparation result for constructible functions
given by Theorem \ref{thm:constrPrep}.  Slivers and the transformation
$H$ from Section \ref{sec:goodCoord}
are used  to prove
Proposition \ref{prop:constrPrep}.

\begin{proposition}
\label{prop:constrPrep}
Consider the situation and notation of Theorem \ref{thm:constrPrep}. Fix $f\in\F$ and a cell $A$ from the decomposition of $X$, and write $f(x) = \sum_{i\in I} T_i(x)$, as in \eqref{eq:sumT}.
 Then the following statements are equivalent:
\begin{renumerate}{\setlength{\itemsep}{3pt}
\item
For all $x_{< n}\in\Pi_{n-1}(A)$, the function $f(x_{<n},\cdot)$ is integrable over the fiber $A_{x_{<n}}$.

\item
The set of the $x_{<n }$ in $\Pi_{n-1}(A)$ for which the function $f(x_{<n},\cdot)$ is integrable over $A_{x_{<n}}$ is dense in $\Pi_{n-1}(A)$.

\item
For all $x_{<n}\in\Pi_{n-1}(A)$ and all $i\in I$, the function
$T_i(x_{<n},\cdot)$ is integrable over $A_{x_{<n}}$.
}
\end{renumerate}
\end{proposition}

\begin{proof}
Statement (i) clearly implies (ii), and (iii) clearly implies (i), thus the only nontrivial implication is that (ii) implies (iii).

So assume (ii). We may assume that $\tld{x}_n$ is unconstrained on $A$, since otherwise there is nothing to prove.  Let $F = G_{A}^{[\theta]}\circ H: B\to A$ be the map constructed by composing the map $G_{A}^{[\theta]}:A_{\bd}\to A$ from Definitions \ref{def:Abd} with the map $H:B\to A_{\bd}$ from Lemma \ref{lem:goodCoord},
and write $y = (y_1,\ldots,y_d)$ for the coordinates on $B$.
For each $i\in I$ and $y\in B$, let
\begin{eqnarray}
\label{eq:Sdef}
S_i(y) & = & T_i(F(y)) \cdot \PD{}{F_n}{y_d}(y).
\end{eqnarray}
Thus $f\circ F(y) \PD{}{F_n}{y_d}(y) = \sum_{i\in I} S_i(y)$ on $B$.  It follows from the construction of $F$  and from the situation given by Theorem \ref{thm:constrPrep} that for each $i\in I$,
\begin{equation}\label{eq:Sjac}
S_i(y) = u_i(y) \prod_{k\in K} y_{k}^{r_k(i)} (\log y_k)^{\ell_k(i)},
\end{equation}
for $K = \lambda^{-1}(J)$, a strong subanalytic unit $u_i$ with center $0$, rational numbers $r_k(i)$, and nonnegative integers $\ell_k(i)$ such that the tuples $(\ell_k(i))_{k\in K}$ are distinct for different $i$ in $I$.  It is enough to show that $\bar{r} > -1$, where
\[
\bar{r} = \min\{r_d(i) : i\in I\}.
\]

Define
\begin{eqnarray*}
I_1 & = &  \{i\in I : r_d(i) = \bar{r}\},\\
\bar{\ell}  & = &   \max\{ \ell_d(i) : i\in I_1\},\\
I_2 & = & \{i\in I_1 : \ell_d(i) = \bar{\ell}\}.
\end{eqnarray*}
By Proposition \ref{lemma:sliver} we may fix a sliver function $\psi':(0,\epsilon)\times(p,q)\to B'$ for $B' = \Pi_{d-1}(B)$, where $(p,q) = (p_2,q_2)\times\cdots\times(p_{d-1},q_{d-1})$.  Write $t = (t_1,\ldots,t_{d-1})$ for a tuple of variables ranging over $(0,\epsilon)\times(p,q)$.  By Lemma \ref{lemma:affine}, by shrinking $(p,q)$, and up to reordering the terms $S_i$, we may assume that there is a positive constant $c$ such that for all $i\in I_2$ with $(r_k(i))_{k\in K}\not = (r_k(1))_{k\in K}$ one has
\begin{equation}\label{eq:I3dominance0}
c + r_1(1) + \sum_{k\in K\setminus\{1\}} r_k(1) t_k
\leq
r_1(i) + \sum_{k\in K\setminus\{1\}} r_k(i) t_k
\end{equation}
on $(p,q)$.
 Define
\begin{eqnarray*}
I_3
  & = &
    \{i\in I_2 : (r_k(i))_{k\in K} = (r_k(1))_{k\in K}\}.
\end{eqnarray*}
 For each $i\in I$ define
\[
\ell'(i) = \sum_{k\in K\setminus \{d\}} \ell_k(i)
\]
and define
\begin{eqnarray*}
\barellprime & = & \max\{\ell'(i) : i\in I_3\}, \\
I_4 & = & \{i\in I_3 : \ell'(i) = \barellprime\}.
\end{eqnarray*}
Define a map $\psi_0 : (0,\epsilon) \times (p,q) \times \RR\to \RR^d$ by
\[
\psi_0(t,y_d) = (\psi'(t), y_d),
\]
and put $B_0:= \psi_0^{-1}(B)$ and
$
\psi: B_0 \to B: (t,y_d)\mapsto \psi_0(t,y_d).
$

Consider the function $W:B\to \RR$ defined by
\[
W(y) = \left(\prod_{k\in K} y_k^{r_k(1)}\right)
    (\log y_1)^{ \barellprime  }
       (\log y_d)^{\bar{\ell}}.
\]
For any function $h:B\to \RR$ and for $(t,y_d)\in B_0$, write $h(t,y_d)$ for $h(\psi(t,y_d))$.
One then has for $(t,y_d)\in B_0$
\begin{equation}\label{eq:W}
W(t,y_d) =  t_{1}^{r_1(1) + \sum_{k\in K\setminus\{1,d\}} r_k(1) t_k}
    (\log t_1)^{ \barellprime  }
    y_{d}^{\bar r}
    (\log y_d)^{\bar{\ell}},
\end{equation}
and for $i\in I$, one has
\begin{eqnarray*}\label{eq:SsliverI3}
S_i(t,y_d)
    & = &
    t_{1}^{ r_1(i) + \sum_{k\in K\setminus\{1,d\}} r_k(i) t_k  }
    (\log t_1)^{\ell'(i)} y_{d}^{r_d(i)}(\log y_d)^{\ell_d(i) }
    \\
    &&
    \cdot\left(\prod_{k\in\{2,\ldots,d-1\}\cap J} t_{k}^{\ell_k(i)}\right)
    u_i (t,y_d)
\end{eqnarray*}
Thus, for $i\in I$  and $(t,y_d)\in B_0$,
\begin{eqnarray}
\label{eq:SW}
\frac{S_i (t,y_d) }{W(t,y_d) }
    & = &
    t_{1}^{r_1(i) - r_1(1) + \sum_{k\in K\setminus\{1,d\}} (r_k(i) - r_k(1))t_k}
    (\log t_1)^{\ell'(i) - \barellprime} \\
    & &
    \cdot \,\, y_{d}^{r_d(i) - \bar{r}}
    (\log y_d)^{\ell_d(i) - \bar{\ell} }
    \left(\prod_{k\in\{2,\ldots,d-1\}\cap J} t_{k}^{\ell_k(i)}\right)
    u_i (t,y_d). \nonumber
\end{eqnarray}
It follows from the definitions of the sets $I_1,\ldots,I_4$ and from \eqref{eq:I3dominance0} and \eqref{eq:SW} that for all $i\in I\setminus I_4$,
\begin{equation}\label{eq:I4dominance}
\lim_{t_1\to 0}
\left(
\lim_{y_d\to 0}
\frac
    {
        S_i(t,y_d)
    }
    {
        W(t,y_d)
    }
\right)
= 0,
\end{equation}
where the limit $t_1\to 0$ is uniform on $(p,q)$, and that for all $i\in I_4$,
\begin{equation}\label{eq:I4class}
\frac{S_i(t,y_d)}{W(t,y_d)}
    =
    \left(\prod_{k\in\{2,\ldots,d-1\}\cap J} t_{k}^{\ell_k(i)}\right)
    u_i (t,y_d).
\end{equation}

Write $u_i(y) = U_i\circ\varphi_i(y)$, where $\varphi_i$ is a bounded rational monomial map on $B$ and $U_i$ is an analytic unit on the closure of the image of $\varphi_i$.  Since $\varphi_i$ is bounded and $y_d$ is unconstrained on $B$ and can approach zero, the powers of $y_d$ that occur in each of the components of $\varphi_i$ must all be nonnegative.  So $\lim_{y_d\to 0}u_i(y) = U_i\circ\varphi_i(y_{<d},0)$, which is a strong subanalytic unit on $B'$, where we have extended $\varphi_i$ naturally on the closure of $B$ in $B'\times\RR$.  Lemma \ref{lemma:sliverMonomialLimit} implies that by shrinking $(p,q)$, we can ensure that $\lim_{t_1\to 0} u_i\circ\psi(t,0) = U_i(0)$ uniformly on $(p,q)$.  In particular, $0$ is in the domain of $U$. In summary,
\begin{equation}\label{eq:u}
\lim_{t_1\to 0}\left( \lim_{y_d\to 0} u_i\circ\psi(t,y_d)\right) = U_i(0),
\end{equation}
where the limit $t_1\to 0$ is uniform on $(p,q)$, and where by the definition of strong subanalytic units, $U_i(0)\not =0$.  By construction, the tuples
\[
(\ell_k(i))_{k \in \{2,\ldots,d-1\}\cap K}
\]
are distinct for different $i$ in $I_4$, so
\begin{equation}\label{eq:tpoly}
\sum_{i\in I_4}
    U_i(0)
    \left(\prod_{k\in\{2,\ldots,d-1\}\cap K} t_{k}^{\ell_k(i)}\right)
\end{equation}
is a nonzero polynomial.  Thus by shrinking $(p,q)$, we can ensure that \eqref{eq:tpoly} is bounded away from $0$ on $(p,q)$.

Equations \eqref{eq:I4dominance}-\eqref{eq:u}, and the fact \eqref{eq:tpoly} is bounded away from $0$ on $(p,q)$, show that by shrinking $\epsilon$ we may ensure that for all $t\in (p,q)$, the limit
\begin{equation}\label{eq:Wasymp}
\lim_{y_d\to 0} \frac{\sum_{i\in I} S_i(t,y_d)}{W(t,y_d)}
\end{equation}
exists and is nonzero.  Since the sliver $\im (\psi)$ is open in $B'$, (ii) of the Proposition implies that the function $y_d\mapsto \sum_{i\in I}S_i(t,y_d)$ is integrable for all $t$ in a dense subset of $(0,\epsilon) \times (p,q)$, and thus by \eqref{eq:Wasymp}, the same is true for $y_d\mapsto W(t,y_d)$.  Therefore  $\bar r>-1$ by \eqref{eq:W}, which finishes the proof.
\end{proof}

%

%% file: LRCfamily_plus1.tex
\section{Proofs of Theorem \ref{thm:LRCfamily} and Proposition  \ref{thm:decayRate}}\label{s:LRCfamily_plus}

In this section we prove Theorem \tprime{\ref{thm:LRCfamily}} and Proposition  \ref{thm:decayRate}.

\begin{proof}[Proof of Theorem \tprime{\ref{THM:LRCfamily}}]
Let $f:X\times Y\times\RR \to\RR$, the set $C$, and the variables $x,y,z$ be as in the statement of the Theorem, and assume that $X\subset\RR^k$ and $Y\subset\RR^n$.  We first prove the Theorem assuming $Y = \RR^n$.  Apply Theorem \ref{thm:constrPrep} to $f$.  This gives a subanalytic cell decomposition $\A$ of $X\times\RR^n\times\RR$, and $\A$ induces a subanalytic cell decomposition of $X\times\RR^n$, namely $\B = \{\Pi_{k+n}(A) : A\in\A\}$.  Consider a cell $A\in\A$ which is fat in $y$, let $B = \Pi_{k+n}(A)$, and fix $x\in\Pi_k(A)$.  Since $B_x$ is open in $\RR^n$ and $C_x$ is dense in $\RR^n$, the set $(C\cap B)_x$ is dense in $B_x$.  The function $z\mapsto f(x,y,z)$ is integrable on $A_{(x,y)}$ for all $y\in (C\cap B)_x$, and hence by Proposition \ref{prop:constrPrep}, for all $y\in B_x$.  Therefore if we define
\[
C' = \bigcup\{D\in\B : \text{$D$ is fat in $y$}\},
\]
the set $C'$ is a subanalytic, $C'_x$ is dense in $\RR^n$ for all $x\in X$, and the linearity of integration implies that for all $x\in X = \Pi_k(C')$ and all $y\in C'_x$, the map $z\mapsto f(x,y,z)$ is integrable on $\RR$.  This completes the proof when $Y = \RR^n$.

To finish we need to show how to reduce to the case that $Y = \RR^n$.  Let $\S$ be a stratification of $Y$ into subanalytic cells, and let $\H$ be the highest level strata of $\S$, namely, the set of all $S\in\S$ which are not contained in the boundary of any other member of $\S$.  Thus $\bigcup\H$ is dense and open in $Y$, and for all $x\in X$ and all $S\in\H$, the set $C_x\cap S$ is dense in $S$.  It therefore suffices to fix $S\in\H$, assume that $C_x$ is a dense subset of $S$ for all $x\in X$, and study the restriction of $f$ to $X\times S\times\RR$.  By projecting into a lower dimensional space, we may assume that $S$ is open in $\RR^n$.  Now extend $f$ by $0$ on $X\times(\RR^n\setminus S)\times\RR$.  Fix $x\in X$.  Because $S$ is open in $\RR^n$ and $f(x,y,z) = 0$ for all $y\not\in S$, the function $z\mapsto f(x,y,z)$ is integrable on $\RR$ for all $y$ in dense subset of $\RR^n$ if and only if it is integrable on $\RR$ for all $y$ in a dense subset of $S$.  So we are done by the case that $Y=\RR^n$.
\end{proof}

\begin{lemma}\label{lemma:subanalBound}
For any $f\in\C(X)$, where $X\subset\RR^n$ is subanalytic, there exists a subanalytic function $h:X\to(0,+\infty)$ such that $|f(x)| \leq h(x)$ for all $x\in X$.
\end{lemma}

\begin{proof}
Let $f\in\C(X)$.  Write $f(x) = \sum_i f_i(x) \prod_j \log f_{i,j}(x)$ for subanalytic functions $f_i:X\to\RR$ and $f_{i,j}:X\to(0,+\infty)$, where the indices $i$ and $j$ run over finite index sets.  The function
$L_+:(0,+\infty)\to(0,+\infty)$ defined by
\[
L_+(t) = \begin{cases}
1/t,    & \text{if $0<t\leq 1$},\\
t,      & \text{if $t > 1$},
\end{cases}
\]
satisfies $|\log t|  < L_+(t)$ for all $t>0$.  Therefore for all $x\in X$,
%
%
\[
|f(x)| \leq   \sum_i |f_i(x)| \prod_j L_+(f_{i,j}(x)) \leq h(x)
\]
for the positively-valued subanalytic function
\[
h(x) = \max\left\{1, \sum_i |f_i(x)| \prod_j L_+(f_{i,j}(x))\right\}.
\]
\end{proof}


\begin{proof}[Proof of Proposition \ref{thm:decayRate}]
Let $f\in\C(X\times\RR)$ for a subanalytic set $X\subset\RR^{n-1}$, and suppose that $\lim_{x_n\to +\infty} f(x_{<n},x_n) = 0$ for all $x_{<n}\in X$.  Our goal is to show that there exist a constant $r>0$ and a subanalytic function $g:X\to(0,+\infty)$ such that
\[
|f(x)| \leq x_{n}^{-r}
\]
for all $x_{<n}\in X$ and all $x_n > g(x_{<n})$.  Let $\A$ be the subanalytic cell decomposition of $X\times\RR$ given by applying Theorem \ref{thm:constrPrep} to $f$.  Fix $A'\in\{\Pi_{n-1}(A) : A\in\A\}$, and fix the unique cell $A\in\A$ of the form $A = \{x : x_{<n}\in A', x_n > a(x_{<n})\}$.  The function $f\Restr{A}$ is of the form given by Theorem \ref{thm:constrPrep}, and we shall henceforth use the notation in the statement of Theorem \ref{thm:constrPrep} without redefinition. Remark \ref{rmk:unbddFiberCell} implies
that $\tld{x}_n = x_n$.  As in the proof of Proposition \ref{prop:constrPrep}, let $F = G_{A}^{[\theta]}\circ H:B\to A$, write $y = (y_1,\ldots,y_d)$ for a tuple of variables ranging over $B$, and write $y' = (y_1,\ldots,y_{d-1})$ and $B' = \Pi_{d-1}(B)$.  Note that $F_n(y) = 1/y_d$.  Therefore $\lim_{y_d\to 0^+} f\circ F(y',y_d) = 0$ for all $y'\in B'$, and it suffices to find an $r > 0$ and a subanalytic function $g:B'\to(0,+\infty)$ such that $|f\circ F(y)| \leq y_{d}^{r}$ for all $y\in B$ with $y_d < g(y')$.

For each $i\in I$ let $S_i(y) = T_i\circ F(y)$, so that $f\circ F(y) = \sum_{i\in I}S_i(y)$ on $B$. Note that the function $S_i(y)$ differs from the notation in the proof of Proposition \ref{prop:constrPrep}, because we do not multiply by $\PD{}{F_n}{y_d}$, but it is of the same form as given by \eqref{eq:Sjac}.  We therefore use the notation of the proof of Proposition \ref{prop:constrPrep} without redefinition, and we proceed as in the proof of the proposition until we get to \eqref{eq:Wasymp}.  At this point, note that since $\lim_{y_d\to 0^+} f\circ F(y',y_d) = 0$ for all $y'\in B'$, and since the limit \eqref{eq:Wasymp} exists and is nonzero, $\lim_{y_d\to 0^+} W(t,y_d) = 0$ for all $t\in (0,\epsilon)\times(p,q)$.  Therefore $\bar{r} > 0$ by \eqref{eq:W}.

Fix $\epsilon > 0$ such that $\bar{r} - \epsilon\,\bar{\ell} > 0$, and fix $r$ such that $0 < r < \bar{r} - \epsilon\,\bar{\ell}$.  The number $r$ may be chosen so that $\bar{r} - \epsilon\,\bar{\ell} - r$ is rational.  Note that $|\log y_d|  < y_{d}^{-\epsilon}$ for all sufficiently small $y_d>0$.  Therefore \eqref{eq:Sjac} shows that there exists a $\delta > 0$ such that for all $y\in B$ with $y_d < \delta$,
\[
|f\circ F(y)| \leq \sum_{i\in I} |u_i(y)| \left(\prod_{k\in K\setminus\{d\}} y_{k}^{r_k(i)} |\log y_k|^{\ell_k(i)}\right) y_{d}^{\bar{r} - \epsilon \bar{\ell}}.
\]
Since the functions $|u_i(y)|$ are units, it follows from Lemma \ref{lemma:subanalBound} that there is a subanalytic function $h:B'\to(0,+\infty)$ such that
\[
|f\circ F(y)| \leq h(y') y_{d}^{\bar{r} - \epsilon \bar{\ell}}
\]
on $B\cap(\RR^{d-1}\times(0,\delta))$.  Define a subanalytic function $g:B'\to (0,+\infty)$ by
\[
g(y') = \min\left\{\delta, \,\, h(y')^{-1/(\bar{r} - \epsilon\bar{\ell} - r)}\right\}.
\]
Then $|f\circ F(y)| \leq y_{d}^{r}$ for all $y\in B$ with $y_d < g(y')$.
\end{proof}

%

%% file: Section8.tex
\section{Proof of Theorem \ref{thm:LRC}}\label{s:LRC}

The presented proof of Theorem \ref{thm:LRC} uses an adapted version of the integration procedure of \cite{LR98}. 
We will describe this procedure in detail since the set-up and generality is different from the one of \cite{LR98}. 
We proceed by induction on $m$, where the integration is performed over $\RR^m$.  The base case of $m=1$  is nontrivial and relies on Theorem \ref{thm:constrPrep} and Proposition \ref{prop:constrPrep}.  In contrast, the induction step will follow rather immediately from the base case by Fubini's theorem and Theorem \tprime{\ref{thm:LRCfamily}}.

\begin{proof}[Proof of Theorem \ref{thm:LRC} for $m=1$]
Let $f\in\C(X\times\RR)$ for a subanalytic set $X\subset\RR^k$.  Our goal is to show that $I_X(f)$ is in $\C(X)$.  We may assume that $\RR\to\RR:y\mapsto f(x,y)$ is integrable on $\RR$ for all $x\in X$.  Apply Theorem \ref{thm:constrPrep} to get a subanalytic cell decomposition of $X\times\RR$ such that on each cell $A$ in the decomposition, $f(x,y)$ can be expressed as a finite sum of terms of the form  \eqref{eq:formT} which, by Proposition \ref{prop:constrPrep}, are integrable over the fiber $A_x$ for all $x\in\Pi_k(A)$.  Because of the linearity of integration, we focus on a single cell $A$ in the decomposition of $X\times\RR$ which is fat in $y$ with base $B$ and on a single term $S(x,y)=T_i(x,y)$ on $A$ of the form \eqref{eq:formT}. It suffices to show that $B\to\RR:x\mapsto\int_{A_x}S(x,y)dy$ is in $\C(B)$. \\

\noindent\emph{Claim.} By partitioning $A$ into smaller subanalytic cylinders (not cells) and by performing subanalytic coordinate transformations, and adjusting $S$ by a Jacobian accordingly, and by the linearity of the integral, we can reduce to the case that $A$ is a cylinder of the form
\begin{equation}\label{eq:Aform}
A = \{(x,y)\in B\times\RR : a(x) < y < b(x)\}
\end{equation}
for subanalytic functions $a,b:B\to\RR$ such that either $0 = a(x) < b(x) \leq \epsilon$ on $B$ or $0 < a(x) < b(x)\leq \epsilon$ on $B$ for some $\epsilon > 0$, and $S$ is of the form
\begin{equation}\label{eq:Sform}
S(x,y) = \left(\sum_{i=1}^{l}S_i(x)y^{-i} + F(S_0(x),y)\right)(\log y)^s
\end{equation}
on $A$, for some $l,s\in\NN$, subanalytic functions $S_1,\ldots,S_l:B\to\RR$ and $S_0:B\to\RR^N$, where $N\in\NN$ and $S_0$ is bounded, and an analytic function $F(X,Y)$ on a neighborhood $V$ of $\cl(S_0(B))\times[-\epsilon,\epsilon]$ which is represented by a power series in $Y$, $\sum_{i=0}^{\infty}F_i(X)Y^i$, which converges on $V$, where $X = (X_1,\ldots,X_N)$ and where $\cl(\cdot)$ denotes the topological closure.   \\

\noindent \emph{Note}: Because $y\mapsto S(x,y)$ is integrable on $A_x$ for all $x\in B$,  necessarily $l = 0$ when $a(x) = 0$ on $B$. \\

We first use the claim to prove the theorem.  Note that
\begin{eqnarray*}
\int_{a(x)}^{b(x)} S(x,y) \,dy
    & = &
    \sum_{i=2}^{l} S_i(x) \int_{a(x)}^{b(x)} y^{-i}(\log y)^s \,dy \\
    & &
    + S_1(x) \int_{a(x)}^{b(x)}\frac{(\log y)^s}{y}\, dy \\
    & &
    + \int_{a(x)}^{b(x)} F(S_0(x),y) (\log y)^s \,dy
\end{eqnarray*}
on $B$.  Clearly,
\[
\int_{a(x)}^{b(x)} \frac{(\log y)^s}{y}\,dy = \left.\frac{1}{s+1}(\log y)^{s+1}\right|_{a(x)}^{b(x)}.
\]
To integrate the other terms, define analytic functions
\begin{eqnarray*}
G(X,Y) & = & \sum_{i=0}^{\infty}\frac{F_i(X)}{i+1}Y^{i+1}, \\
H(X,Y) & = & \sum_{i=0}^{\infty}\frac{F_i(X)}{i+1}Y^i,
\end{eqnarray*}
on $V$, and note that $\PD{}{G}{Y}(X,Y) = F(X,Y)$ and $G(X,Y) = Y H(X,Y)$.  The theorem now follows by inducting on $s$, using the fact that
when $s=0$,
\begin{eqnarray*}
\sum_{i=2}^{l} S_i(x) \int_{a(x)}^{b(x)} y^{-i} \,dy
    & = &
    \sum_{i=2}^{l} \left.\frac{S_i(x)}{-i+1}y^{-i+1}\right|_{a(x)}^{b(x)}, \\
\int_{a(x)}^{b(x)} F(S_0(x),y) \,dy
    & = &
    G(S_0(x),y)\Big|_{a(x)}^{b(x)},
\end{eqnarray*}
and when $s > 0$, integration by parts gives
\begin{eqnarray*}
\int_{a(x)}^{b(x)} y^{-i} (\log y)^s\,dy
    & = &
    \left.\frac{1}{-i+1}y^{-i+1}(\log y)^s\right|_{a(x)}^{b(x)} - \frac{s}{-i+1}\int_{a(x)}^{b(x)} y^{-i}(\log y)^{s-1}dy,
    \\
\int_{a(x)}^{b(x)} F(S_0(x),y) (\log y)^s \,dy
    & = &
    G(S_0(x),y)(\log y)^s\Big|_{a(x)}^{b(x)} - s\int_{a(x)}^{b(x)} H(S_0(x),y) (\log y)^{s-1} \, dy.
\end{eqnarray*}
\hfill\\

We now prove the claim.  We may suppose that $A$ is an $(n+1)$-dimensional $\lambda$-cell with center $\theta$.  By applying the map $G^{[\theta]}_A$ of Definitions \ref{def:Abd}  and adjusting $S$ by a Jacobian, we may assume that $A$ an open subanalytic cell in $(0,1)^{n+1}$ with center $0$ which is of the form
\begin{equation}\label{eq:A}
A = \{(x,y)\in B\times\RR : a(x) < y < b(x)\}
\end{equation}
for subanalytic terms $a$ and $b$ on $B$, where $a(x) < b(x) $ on $B$ and either $a(x) = 0$ on $B$ or $a(x) > 0$ on $B$, and that
\begin{equation}\label{eq:S}
S(x,y) = g(x) y^r (\log y)^s u(x,y)
\end{equation}
on $A$, where $g\in\C(B)$, $r\in\QQ$, $s\in\NN$, and $u(x,y)$ is a strong subanalytic unit on $A$ with center $0$. We are done if $g(x)$ is identically equal to $0$ on $B$, so assume otherwise. By the linearity of the integral we may pull $g(x)$ out of the integral and thus suppose that $g(x)=1$.  By applying Lemma \ref{lemma:strongUnit}, which partitions $B$ into smaller subanalytic sets and thereby partitions $A$ into smaller subanalytic cylinders (not cells), we may assume that $A$ is a subanalytic cylinder of the form \eqref{eq:A}, and that $S(x,y)$ is of the form \eqref{eq:S} but with $u(x,y) = U\circ\varphi(x,y)$, where $\varphi:A\to\RR^{N+2}$ is a bounded function of the form
\[
\varphi(x,y)=(c_1(x),\ldots,c_N(x),c_{N+1}(x)y^{1/p},c_{N+2}(x)y^{-1/p})
\]
for analytic subanalytic terms $c_1,\ldots,c_{N+2}$ on $B$ and a positive integer $p$ such that $pr$ is an integer, and $U:\RR^{N+2}\to\RR$ is an analytic unit on the closure of the image of $\varphi$.  Let $(X,Y,Z) =(X_1,\ldots,X_N,Y,Z)$ denote a tuple of variables ranging over the domain of $U$, and define $c:B\to\RR^N$ by $c(x) = (c_1(x),\ldots,c_N(x))$.
 By further partitioning $B$ we may assume that $c_{N+1}(x) = 0$ on $B$ or that $c_{N+1}(x) \neq 0$ on $B$.  If $c_{N+1}(x) = 0$ on $B$ then apply the change of variables $y\mapsto y^p$, and if $c_{N+1}(x)\neq 0$ on $B$ then apply the change of variables $y\mapsto (y/c_{N+1}(x))^p$.  In either case we may adjust the definition of $S$ by a Jacobian to assume that $r$ is an integer and that
\begin{equation}\label{eq:varphi}
\varphi(x,y) = (c(x),y,d(x)/y)
\end{equation}
for a subanalytic term $d$ on $B$.  By further partitioning $B$ and absorbing sign information of $d(x)$ into $U$, we may assume that $d(x) = 0$ on $B$ or $d(x) > 0$ on $B$.

Let $K$ be the closure of $\{(y,d(x)/y) : (x,y)\in A\}$.  The set $K$ is compact, and for each $(Y_0,Z_0)\in K$ and $\epsilon > 0$, the set $\{(x,y)\in A : |y - Y_0| < \epsilon, |d(x)/y - Z_0| < \epsilon\}$ can be partitioned into finitely many subanalytic cylinders.  So it suffices to fix $(Y_0,Z_0)\in K$, and we may assume that for all $(x,y)\in A$,
\begin{eqnarray}
\label{eq:Y0}
|y - Y_0| & < & \epsilon, \\
\label{eq:Z0}
\left|\frac{d(x)}{y} - Z_0\right| & < & \epsilon,
\end{eqnarray}
where $\epsilon > 0$ can be chosen to be as small as we wish.  Note that if after performing a change of variable of the form $(x,y)\mapsto (x, h(x)(y + y_0))$ for some subanalytic function $h(x)$ and constant $y_0$, we can express $u(x,y)$ in the form $U(c(x),y)$ with $|y|$ small (for a new choice of $U(X,Y)$ and $c(x)$), then $S(x,y)$ will be of the form \eqref{eq:Sform} in these new variables.

If $d(x) = 0$, then we are in this form by performing the change of variables $y\mapsto y + Y_0$.  If $Y_0\neq 0$, then $d(x)/y$ is analytic in $y$, and we are in this form by the same change of variables $y \mapsto y + Y_0$.  So we may assume that $Y_0 = 0$ and $d(x) > 0$ on $B$.

Suppose that $Z_0 \neq 0$.
Perform the change of variables $y \mapsto \frac{d(x)(y+1)}{Z_0}$. Then $d(x)/y$ in the old variables becomes $\frac{Z_0}{1+y}$ in the new variables.  Thus \eqref{eq:Z0} becomes $|\frac{Z_0}{y + 1} - Z_0| < \epsilon$, which means that $|y|$ is small if $\epsilon$ is small.  So $\frac{Z_0}{y+1}$ is analytic in $y$, and so also $U\left(c(x),\frac{d(x)}{Z_0}(y+1),\frac{Z_0}{y+1}\right)$, and we are done.

Now suppose that $Z_0 = 0$.  Define
\[
F(X,Y,Z) =
\begin{cases}
Y^r U(X,Y,Z),       & \text{if $r\geq 0$}, \\
Z^{-r} U(X,Y,Z),    & \text{if $r < 0$}.
\end{cases}
\]
By pulling $d(x)^r$ out of the integral if $r < 0$, we may assume that
\[
S(x,y) = (\log y)^s F(c(x),y,d(x)/y).
\]
Note that $F$ is analytic on the domain of $U$.  We now use Lion and Rolin's ``splitting lemma'', which can be proven by writing $F$ as a doubly-indexed power series in $Y$ and $Z$, say $F(X,Y,Z) = \sum_{(i,j)\in\NN^2} F_{i,j}(X) Y^i Z^j$, and then splitting this sum into the three sums $\sum_{i-j\leq -2}F_{i,j}(X)Y^i Z^j$, $\sum_{i-j = -1}F_{i,j}(X)Y^i Z^j$, and $\sum_{i-j\geq 0}F_{i,j}(X)Y^i Z^j$. \\

\noindent\emph{The Splitting Lemma}:
There exist $\epsilon \in (0,1)$ and functions $F_{\leq -2}$, $F_{-1}$ and $F_{\geq 0}$ which are analytic on $\cl(c(B))\times[-\epsilon,\epsilon]^2$, $\cl(c(B))\times[-\epsilon,\epsilon]$ and $\cl(c(B))\times[-\epsilon,\epsilon]^2$, respectively, such that
\begin{eqnarray*}
F(c(x),y,d(x)/y)
  & = &
    \left(\frac{d(x)}{y}\right)^2
    F_{\leq -2}\left(c(x),d(x),\frac{d(x)}{y}\right) \\
  & &
    + \left(\frac{d(x)}{y}\right) F_{-1}(c(x),d(x)) \\
  & &
    + F_{\geq 0}(c(x),d(x),y)
\end{eqnarray*}
on $\{(x,y)\in A : \text{$|y| < \epsilon$ and $|d(x)/y| < \epsilon$}\}$. \\

To finish, use the splitting lemma to express $S(x,y)$ as the sum of the function
\[
\left(F_{-1}(c(x),d(x))\left(\frac{d(x)}{y}\right) + F_{\geq 0}(c(x),d(x),y)\right)(\log y)^s,
\]
which is in the form $\eqref{eq:Sform}$, and the function
\[
(\log y)^s \left(\frac{d(x)}{y}\right)^2 F_{\leq -2}\left(c(x),d(x), \frac{d(x)}{y}\right),
\]
which can be reduced to the form \eqref{eq:Sform} using the change of variables $y \mapsto d(x)/y$ and adjusting by a Jacobian.
\end{proof}

\begin{proof}[Proof of Theorem \ref{thm:LRC} for a general value of $m$]
Let $m\geq 1$, and let $f\in\C(X\times\RR^m)$ for a subanalytic set $X\subset\RR^k$.  Our goal is to show that $I_X(f)$ is in $\C(X)$.  We may assume that $\RR^m\to\RR:y\mapsto f(x,y)$ is integrable on $\RR^m$ for all $x\in X$.  The case of $m=1$ has been proven, so let $m > 1$.  By Fubini's theorem and Theorem \tprime{\ref{THM:LRCfamily}}, there exists a subanalytic family $\{C_x\}_{x\in X}$ of dense subsets of $\RR^{m-1}$ such that for all $x\in X$ and all $y_{<m}\in C_x$, the function $y_m\mapsto f(x,y_{<m}, y_m)$ is integrable over $\RR$. Replacing $f(x,y)$ with the product of $f$ and the characteristic function of $\{(x,y_{<m}) : y_{<m}\in C_x\}$ does not affect the function $I_X(f)$.  We may therefore assume that $y_m\mapsto f(x,y)$ is integrable on $\RR$ for all $(x,y_{<m})\in X\times\RR^{m-1}$.  The base case of our induction shows that $g = I_{X\times\RR^{m-1}}(f)$ is in $\C(X\times\RR^{m-1})$.  By Fubini's theorem, $y_{<m}\mapsto g(x,y_{<m})$ is integrable over $\RR^{m-1}$ for all $x\in X$ and moreover $I_X(f) = I_X(g)$.  The induction hypothesis then shows that $I_X(g)$ is in $\C(X)$, which completes the proof of Theorem \ref{thm:LRC}.
\end{proof}